\documentclass[reqno,12pt]{elsarticle}%
\usepackage{amsthm,amsmath}
\usepackage{enumerate}
\usepackage{fancyhdr}
\usepackage{amssymb}
\usepackage[pdfpagemode=None, 
pdfstartview=FitB]{hyperref}
\usepackage[body={36cc,51cc}, hmargin={2.8cm,2.8cm}]{geometry}
\usepackage{appendix}
\usepackage{graphicx}
\DeclareMathSizes{12}{11}{7}{5}

\newtheorem{theorem}{Theorem}[section]
 
 \newtheorem{lemma}[theorem]{Lemma}
 \newtheorem{proposition}[theorem]{Proposition}

 \theoremstyle{definition}
 \newtheorem{definition}[theorem]{Definition}
 \theoremstyle{definition}

 \newtheorem{remark}[theorem]{Remark}
 \theoremstyle{remark}

 \newcommand{\p}{\partial}

 \newcommand{\N}{\mathbb{N}}
 
 \newcommand{\norm}[1]{\big\Vert#1\big\Vert}
 
 \newcommand{\abs}[1]{\left\vert#1\right\vert}
 \newcommand{\set}[1]{\left\{#1\right\}}
 
 \newcommand{\inner}[1]{\left(#1\right)}

 \newcommand{\com}[1]{\big[#1\big]}
 \newcommand{\reff}[1]{(\ref{#1})}

\begin{document}

\begin{frontmatter}

\title{Regularity  of  traveling free surface water waves with
  vorticity}

\author[1]{Hua Chen}
\ead{chenhua@whu.edu.cn}
\author[1]{Wei-Xi Li}
\ead{wei-xi.li@whu.edu.cn}
\author[2]{Ling-Jun Wang}
\ead{wanglingjun@wust.edu.cn}

\address[1]{School of Mathematics and Statistics, Wuhan
   University, 430072 Wuhan, China}
 \address[2]{
              School of Science, Wuhan University of Science and
   Technology, 430065 Wuhan, China}

\begin{abstract}
We prove real analyticity of all the streamlines, including the free
surface, of a gravity-  or  capillary-gravity-driven steady flow of
water over a flat bed, with a H\"{o}lder continuous
vorticity function,  provided that the propagating speed of the wave on
the free surface exceeds the horizontal fluid velocity throughout the
flow.
Furthermore, if the
vorticity possesses  some Gevrey regularity of  index $s$,  then the stream
function  admits the same Gevrey regularity throughout the fluid
domain;  in particular if the Gevrey index $s$ equals to $1$,  then we
obtain  analyticity of the stream function.  The regularity results  hold for both
periodic and solitary water waves.
\end{abstract}


\end{frontmatter}

\section{Introduction}

Recently, water waves with vorticity, also called rotational waves,
are investigated extensively. There have been a series of works
concerning rotational waves, including existence results for small-
and large-amplitude waves
\cite{MR2027299,MR2349872,MR2454604,MR2215274,MR2385741,MR2262949}, as
well as results on uniqueness and symmetry, analyticity of wave
profile \cite{MR2753609,MR2763714,HenryJmfm,MaImrn,MaQam}, and so
on. The present  work is mainly concerned with the analyticity or regularity results for
rotational water waves, with or without surface tension.

Assuming that the vorticity function is H\"{o}lder continuously differentiable, Constantin
and Strauss \cite{MR2027299} proved, by using methods of bifurcation theory, the existence
of global bifurcation branches consisting of periodic water waves which travel above a flat bottom with constant speed exceeding that of the water particles enclosed by the wave. The assumption that the wave speed exceeds that of the water particles is supported by field evidence \cite{MR1872073}, and means that the waves are not near breaking or stagnation. We consider such waves as well in this paper.

In the irrotational setting, a classical result due to Lewy \cite{MR0049399} showed that irrotational waves without stagnation points have real analytic profiles, by use of a generalized Schwartz reflection principle. Recently,
Constantin and Escher \cite{MR2753609} generalized this result to rotational case, and proved that, under the same assumption  on the vorticity function as in \cite{MR2027299}, namely H\"{o}lder continuity of the first derivative, each streamline,  except the free surface,  is real analytic;
if further  the vorticity function is real analytic,  then the free surface
itself is also analytic.    The arguments in \cite{MR2753609} base
on translational invariance property of the resulting operator in the
direction of wave propagation, and the celebrated result due to Kinderlehrer et al. \cite{MR531272} on regularity for elliptic
free boundary problems.  Later on, similar results as in
\cite{MR2753609} are obtained for deep-water waves  \cite{MaQam},
flows with merely bounded vorticity \cite{MaImrn},  solitary-water
waves \cite{HurImrn}, and for periodic capillary-gravity waves
\cite{MR2763714,MR2769902,HenryJmfm} where it was shown that the wave
profile is furthermore $C^\infty$-smooth if the vorticity function is
H\"{o}lder continuously differentiable.   Note that in  the aforementioned works
 the analyticity of free surface   is  established  under
the extra assumption that the vorticity function is analytic.

It is natural to expect the
analyticity of the free surface for flows  with only H\"{o}lder
continuous  vorticity.   This is what we will do in this work.   Precisely,  assuming that the
vorticity function is only H\"{o}lder continuous, we
obtain the real analyticity of all the streamlines, including the free
surface,  of the steady flow
over a flat bed in the absence of stagnation points.   As in the above works, we first use  an appropriate hodograph change
of variable that transforms the free boundary value problem (corresponding in
a frame moving at the constant wave speed to the governing equations for water
waves with vorticity) into a nonlinear boundary problem for a quasi-linear
elliptic equation in a fixed rectangular domain. Then basing on some a priori Schauder estimates
 (see for instance  \cite[Theorem 6.30]{MR1814364},  and  \cite{MR0125307} for general
nonlinear elliptic equations with nonlinear oblique boundary
conditions), we show the analyticity of streamlines by giving successively a quantitative bound for
each derivative of the streamlines in the H\"{o}lder norm.

We also study the case when the
vorticity possesses more regularity property rather than H\"{o}lder
continuity,  namely
Gevrey regularity of index $s$. Gevrey class  is an
intermediate space between the spaces of smooth functions and analytic
functions,  and the Gevrey class function of index $1$ is just the
real-analytic function;  see Subsection \ref{subsec result} below for precise definition of Gevrey
class. In this case we investigate Gevrey regularity of stream
function throughout the fluid domain.
If the vorticity is Gevrey regular, we prove that  the stream
function  admits  the same Gevrey regularity in the fluid domain, up to
the free surface;   see Theorem \ref{th2} stated in Subsection \ref{subsec result}.       To obtain this,
we firstly establish the corresponding  regularity for the height function
in a fixed rectangular domain, and then use the result of
\cite[Theorem 3.1]{clx2011} to
show that the Gevrey regularity is preserved through   partial hodograph
transformation.

We conclude this introduction by pointing out that our approach
 applies for both  periodic and  solitary waves.  For
 simplicity  we consider in this work only flows with finite depth.
 With suitable modifications, the methods may be
employed to the periodic waves on deep water with vorticity,  constructed in \cite{MR2215274,MR2788362}.

This paper is organized as follows. In Section \ref{sec2} we
formulate the rotational water-wave problem as  free boundary problem
for stream function and its equivalent reformulation in a fixed
rectangular domain,  and state our main regularity results.    Notations and
some useful inequalities are listed.  Section \ref{sec3}
is devoted to the proof of analyticity of  streamlines including the
free surface.  In Section \ref{sec4} we study the Gevrey (analytic)
regularity of stream function.  In the last section,
Section \ref{sec5}, we consider the travelling capillary-gravity water
waves, and obtain similar regularity
results for streamlines and stream function.

\section{Preliminaries and  main results}\label{sec2}

\subsection{The governing equations for rotational water waves}
Consider a steady two-dimensional flow of an incompressible inviscid fluid over a rigid flat bed $y=-d$ with $0<d<\infty$, acted upon by gravity,
and a steady wave on the free surface of the flow. By steady, we mean that the flow and the surface wave move at a constant speed
from left to right without changing their configuration. In the frame
of reference moving at the wave speed $c>0$,  let the $x$-axis
point in the direction of wave propagation, the free surface
be given in the graph form by $y= \eta(x)$ and let the liquid occupy the stationary domain
\[
  \Omega=\{(x,y)\in \mathbb{R}^2: -d<y<\eta(x)\}.
\]
Take $y = 0$ to represent the location of the undisturbed water
surface.
Let $(u(x,y),v(x,y))$ denote the velocity field, and define the \textit{stream function} $\psi(x,y)$ by $\psi(x,\eta(x))=0$ and
\begin{equation}\label{def psi}
  \psi_y=u-c,\quad \psi_x=-v.
\end{equation}
The flow is allowed to be rotational and characterized by the vorticity
$\omega=v_x-u_y$.
Consider also only waves that are not near breaking or stagnation, so that
\begin{equation}\label{no stag}
  \psi_y(x,y)\leq -\delta<0\quad {\rm in}~~  \bar{\Omega}
\end{equation} for some $\delta>0$, which  implies that the vorticity $\omega$ is globally a function of the stream function $\psi$, denoted by $\gamma(-\psi)$; see \cite{MR2027299}. The governing equations for the gravity water wave problem are formulated as
\begin{subequations}\label{EquPsi}
\begin{eqnarray}
 \triangle\psi=-\gamma(-\psi),&\quad (x,y)\in \Omega,\label{EquPsi1}\\
 \abs{\nabla \psi}^2+2g(y+d)=Q, &\quad y=\eta(x),\label{EquPsi2}\\
  \psi=0,&\quad y=\eta(x), \label{EquPsi3}\\
  \psi=-p_0,& \quad y=-d.  \label{EquPsi4}
 \end{eqnarray}
\end{subequations}
Here $g>0$ is the gravitational constant of acceleration, $Q$ is  a constant related to the energy and
\begin{eqnarray*}
  p_0=\int_{-d}^{\eta(x)}\psi_y(x,y)~dy<0
\end{eqnarray*}
is the relative mass flux (independent of $x$). Moreover the \textit{wave profile} $\eta(x)$ represents an unknown in the problem since it is a free surface. We refer to \cite{MR2027299} for the detailed derivation of the above system of governing equations.

The level sets $\{(x,y):\psi(x,y)={\rm constant}\}$ are \textit{streamlines} of the fluid motion. Note that the free surface and the rigid bottom
are themselves streamlines in virtue of \reff{EquPsi3} and
\reff{EquPsi4}. Observing \reff{no stag}, each streamline
$\psi(x,y)=p$, with $p\in[p_0,0]$, can be described by the graph of
some function $y=\sigma_p(x)$.

\subsection{Statement of the main results}\label{subsec result}

To state our main results, we first recall the definition of Gevrey
class functions, which is an intermediate space between the spaces of
smooth functions and real-analytic functions;  see \cite[Chapter 1]{MR1249275}
for more detail.

\begin{definition}\label{def gevrey}
 Let $W$ be an open subset of $\mathbb{R}^d$ and $f$
be a real-valued  function defined on the closure $\bar W$ of $W$.  We say $f$ belongs to  Gevrey class
in $\bar {W}$ of index $s\geq 1$,  denoted by $f\in G^s(\bar W)$,   if $f\in
C^\infty(\bar W)$  and for any compact subset $K$ of  $ \bar W$  there
exists a constant $C_K$, depending only on $K$,
such that  
\[
  \forall~\alpha\in\mathbb N^d,\quad
  \max_{x\in K}\abs{\partial^\alpha f(x)}\leq C_K^{\abs\alpha+1}\inner{\abs{\alpha}!}^s,
\]
where $\abs\alpha=\alpha_1+\alpha_2+\cdots+\alpha_d.$
\end{definition}

In particular   $G^1(\bar W)$ is the space  of all real analytic functions
in $\bar W$.

Throughout the paper let  $C^{k,\mu}(\bar W)$,
$k\in\mathbb N, \mu\in(0,1)$,   be
the  standard H\"{o}lder  space of functions $f : \bar W \rightarrow \mathbb R$ with H\"{o}lder-continuous derivatives
of exponent $\mu$ up to order $k$.      For given $p_0<0$ and
$\gamma\in C^{1,\mu}([p_0,0])$,  the existence of periodic and
supercritical small-amplitude solitary water waves has been
established in \cite{MR2027299} and \cite{MR2454604,MR2385741},
respectively. Our main result below shows that, with a H\"{o}lder continuous
vorticity,
each streamline can be described by the graph of some analytic function.

\begin{theorem}\label{th1}
Let the function   $\gamma$ in  \reff{EquPsi1}   belong to the H\"{o}lder space  $C^{0,\mu}([p_0, 0])$  with  $p_0<0$ and $0<\mu<1$  given,  and  let $\psi(x,y)\in
C^{3,\mu}(\bar\Omega)$  be the stream function for the boundary problem
\reff{EquPsi1}-\reff{EquPsi4} with free surface $y=\eta(x)$.  Suppose 
$\psi$ satisfies the no-stagnation assumption \reff{no stag}. Then each
streamline including the free surface $y=\eta(x)$ is a real-analytic curve.
\end{theorem}

\begin{remark}
The existence of the stream function $\psi$ for the boundary problem
\reff{EquPsi1}-\reff{EquPsi4}  is well-known (cf. \cite{MR2027299}).
\end{remark}

The following result shows that the stream function admits the same
regularity as the vorticity.

\begin{theorem}\label{th2}
Under the same assumptions as in Theorem \ref{th1},  if  $\gamma \in
G^s([p_0, 0])$ additionally with  $s\geq1$ given,  then we have $\psi(x,y) \in
G^s(\bar \Omega )$;  in particular if $s=1$, i.e., $\gamma$ is analytic in $[p_0,
0]$, then the stream function $\psi(x,y)$ is analytic in $\bar\Omega$.
\end{theorem}

\begin{remark}
The above results  also hold for the travelling capillary-gravity
water waves; see Theorem \ref{th3} in Section \ref{sec5}.
\end{remark}

\subsection{Reformulation}

Under the no-stagnation assumption \reff{no stag}, we  can use the
partial hodograph change of variables  to transform  the free boundary problem \reff{EquPsi1}-\reff{EquPsi4} into a problem with fixed boundary.
Precisely,  if we  introduce the new variable $(q,p)$ with
  \[
   q=x, \quad p=-\psi(x,y),
\]
and exchange  the roles of the $y$-coordinate and
$\psi$ by setting
\begin{eqnarray*}
  h(q,p)=y+d,
\end{eqnarray*}
then the fluid domain $\Omega$ is transformed into a fixed infinite strip
\begin{eqnarray*}
   R=\{(q,p): q\in\mathbb R,\ p_0<p<0\},
\end{eqnarray*}
and the system \reff{EquPsi1}-\reff{EquPsi4} can be reformulated  in this strip as
\begin{subequations}
\begin{eqnarray}
 (1+h_q^2)h_{pp}-2h_ph_q h_{pq}+h_p^2h_{qq}+\gamma(p)h_p^3=0,& {\rm in~~} R,\label{Equh1}\\
  1+h_q^2+(2gh-Q)h_p^2=0,&\quad {\rm on~~} p=0,\label{Equh2}\\
  h=0,&\quad {\rm on~~} p=p_0.\label{Equh3}
\end{eqnarray}
\end{subequations}
We refer to \cite{MR2027299}  for the equivalence of the two  systems
\reff{EquPsi1}-\reff{EquPsi4} and \reff{Equh1}-\reff{Equh3} of governing equations.
Note that $h_p=\frac{1}{c-u}$. The no-stagnation assumption \reff{no stag} ensures that
\begin{eqnarray}\label{hp bounded}
 0<\inf_{(q,p)\in\bar R}h_p\leq h_p\leq \sup_{(q,p)\in\bar R}h_p\leq\frac{1}{\delta}.
\end{eqnarray}

 The following proposition shows that the regularity is preserved
through hodograph transformation.  So we only need to study  the
above  problem \reff{Equh1}-\reff{Equh3} instead of
the original one \reff{EquPsi1}-\reff{EquPsi4}.

\begin{proposition}\label{reserv}
      Let $h\in C^{2,\mu}(\bar R)$ be  a solution to the
      problem \reff{Equh1}-\reff{Equh3} .  If the mapping  $q\mapsto
h(q,p)$,  with any fixed $p\in[p_0,0]$,  is  analytic in  $\mathbb R$,
then each  streamline including the free surface is an analytic
curve.   Moreover if  $h\in G^s(\bar R)$ then  the
      stream function $\psi$ for \reff{EquPsi1}-\reff{EquPsi4}  lies in
      $G^s(\bar\Omega)$; in particular $\psi$ is analytic in
      $\bar\Omega$ provided $h$ is analytic in $\bar R$.
\end{proposition}

\begin{proof}
  The first statement is straightforward.  Indeed, Observing \reff{no stag}, each streamline
$\psi(x,y)=p$, with fixed $p\in[p_0,0]$, can be described by the graph of
some function $y=\sigma_p(x)$.  The analyticity of  $x\mapsto\sigma_p(x)$
follows at once from the analyticity of  the mapping  $q\mapsto
h(q,p)$,  due to  the partial hodograph change of variables. 

 As for the second one,  we  rewrite the hodograph transform as
  \[
     q=x, ~~p=\tilde \psi_y
\]
with $\tilde\psi(x,y)\stackrel{\rm def}{=}-\int_0^y\psi(x,z)dz$. This
is just the  classic partial Legendre transformation.  If $h(q, p)\in
G^s(\bar R)$  then $y=y(q, p)\in
G^s(\bar R)$. Thus
 by
\cite[Theorem 3.1]{clx2011},  we have   $\tilde \psi\in
G^s(\bar\Omega)$ and thus $\psi\in
G^s(\bar\Omega)$ since $G^s(\bar\Omega)$ is closed under
differentiation.
\end{proof}

\subsection{Notations and some useful inequalities}
We list some notations and useful inequalities which will be used
throughout the paper.   Let $k\in\N$ and $\mu\in(0,1)$,  and
let
$\inner{C^{k,\mu}(\bar R);
  \norm{\cdot}_{k,\mu;\bar R}}$ be  the standard H\"{o}lder space equipped
with the norm
\begin{eqnarray*}
  \norm{w}_{k,\mu;\bar R}=\sum_{\abs{\alpha}=0}^{k}\sup_{\bar R}\abs{\p^{\alpha}w(q,p)}+\sup_{\abs{\alpha}=k}\sup_{\stackrel{(q,p)\neq(\tilde q,\tilde p)}{ \bar R}}\frac{\abs{\p^{\alpha}w(q,p)-\p^\alpha w(\tilde q,\tilde p)}}{\abs{(q,p)-(\tilde q,\tilde p)}^\mu}.
\end{eqnarray*}
To simplify the notation we will use the notation
$\norm{\cdot}_{k,\mu}$ instead of $\norm{\cdot}_{k,\mu;\bar R}$
if no confusion occurs.  For the case when $\mu=0$, we naturally  define
\[
  \norm{w}_{k}=\norm{w}_{k,0}=\sum_{\abs{\alpha}=0}^{k}\sup_{\bar R}\abs{\p^{\alpha}w}.
\]
For $\mu\in(0,1)$, direct verification shows that
\begin{eqnarray}\label{algebra}
  \norm{uw}_{0,\mu}\leq \norm{u}_{0,\mu}\norm{w}_{0,\mu},\quad \norm{uw}_{1,\mu}\leq 2\norm{u}_{1,\mu}\norm{w}_{1,\mu}.
\end{eqnarray}

For a multi-index $\alpha=(\alpha_1,\alpha_2)\in\N^2$, we denote
$\p^\alpha=\p_q^{\alpha_1}\p_p^{\alpha_2}$,~  $\alpha!=\alpha_1!\alpha_2!$ and denote the length of $\alpha$ by $\abs{\alpha}=\alpha_1+\alpha_2$. Moreover for two multi-indices $\alpha$ and $\beta=(\beta_1,\beta_2)\in\N^2$, by
$\beta\leq \alpha $ we mean $\beta_i\leq \alpha_i$ for each $1\leq
i\leq 2$.  Let ${\alpha\choose\beta}$ be the binomial coefficient, i.e.,
\[
  {\alpha\choose\beta}=\frac{\alpha!}{\beta!(\alpha-\beta)!}=\frac{\alpha_1!\alpha_2!}{\beta_1!(\alpha_1-\beta_1)!\beta_2!(\alpha_2-\beta_2)!}.
\]
In the sequel, we use the convention that  $m!=1$ if $m\leq 0$.

\section{Analyticity of streamlines}\label{sec analyticity}\label{sec3}

We prove in this section  the analyticity of streamlines, including the free
surface $y=\eta(x)$.    In view of  Proposition \ref{reserv},  it suffices to show the following
conclusion  that the  map $q\mapsto h(q,p)$ is
analytic for all $p\in[p_0,0]$.

\begin{proposition}\label{propassu}
   Let $\gamma\in C^{0,\mu}\inner{[p_0, 0]}$ with $p_0<0$ and $0<\mu<1$
   given,  and $h\in C^{2,\mu}(\bar
   R)$ be a solution of the governing equations
   \reff{Equh1}-\reff{Equh3}.
Then there exists a constant $L\geq 1$,  such that for all $m\in\N$
with $m\geq 2$, one has the following estimate
  \begin{equation}\label{Em}
   (E_m):\quad\quad \norm{\p^m_q h}_{2,\mu}\leq L^{m-1}(m-2)!.
  \end{equation}
Thus
 the map $q\mapsto h(q,p)$ is analytic for all $p\in[p_0,0]$.
\end{proposition}

\begin{remark}
 As to be seen in the proof below, the constant $L$ depends  on $\mu,
\inf_{\bar R}h_p$,  $ \norm{h}_{2,\mu}$, $\norm{\gamma}_{0,\mu}$ and the number $\delta$ given in \reff{hp
bounded}, but
 independent of the order $m$ of derivative.
\end{remark}

\begin{remark}\label{remReg}
Starting from the $C^{2,\mu}$-regularity  solution $h$ of the governing equations
   \reff{Equh1}-\reff{Equh3},  we use the  Schauder estimate (
  cf. \cite[Theorem 6.30]{MR1814364}) for  $\p_q h$ which satisfies  a
  nonlinear elliptic equation of the same type as \reff{Equh1}-\reff{Equh3}, to conclude that  $\p_q h\in
  C^{2,\mu}(\bar R)$.   Repeating the procedure, we can derive by
  standard  iteration that $\p_q^k h\in
  C^{2,\mu}(\bar R)$  for any $k\in\N$;
  see for instance \cite{MR2027299,HenryJmfm}.
\end{remark}

To confirm the last statement in the above Proposition \ref{propassu},
we   choose $C$ in such a way
that
\[
     C=\max\set{L,  \norm{h}_{1,\mu}},
\]
which,  along with the estimate $(E_m)$ with $m\geq 2$   in Proposition \ref{propassu},  yields
\begin{eqnarray*}
\forall~m\in\mathbb N, \quad \max_{(q,p)\in
  \bar R}|\partial_q^m{h}(q,p)| \leq C^{m+1}m! .
\end{eqnarray*}
In particular, for any  $p\in[p_0,0]$,  $\max_{q\in\mathbb R} \abs{\p_q^m h(q,p)}\leq C^{m+1}
m!$.  This gives the real analyticity of the map $q\mapsto h(q,p)$,
$p\in[p_0,0]$.  

Before proving the above proposition, we first give the following
technical  lemma, and present its proof  at the end of this section.

\begin{lemma}\label{stab}

Let $\ell=1$ or $2$ be given, and let $\norm{\cdot}$ stand for some
H\"{o}lder norm $\norm{\cdot}_{0,\mu}$ or  $\norm{\cdot}_{1,\mu}$. Suppose that $k_0$ is an integer
with $k_0\geq\ell+1$,  and $\p_q^k u_j\in
C^{0,\mu}(\bar R)$ for all $k\leq k_0$,  $j=1,2,3$.  If there
exists  a constant $H\geq 1$ such that
\begin{equation}\label{condition3}
 \forall~\ell+1\leq k\leq k_0,  \quad  \|\p_q^k u_j\|
  \leq H^{k-\ell}(k-\ell-1)!, \quad j=1,2,3,
\end{equation}
then we can find a constant $C_*$ depending only on $\ell$ such that
\begin{eqnarray*}
 \forall~\ell+1\leq k\leq k_0,  \quad
 \big\|\p_q^k\inner{u_1u_2u_3}\big\|\leq C_* \Big(\sum_{j=1}^3\norm{u_j}_{\ell+1,\mu}+1\Big)^6 H^{k-\ell}(k-\ell-1)!.
\end{eqnarray*}

\end{lemma}

We now prove Proposition \ref{propassu}.

\begin{proof}[Proof of Proposition \ref{propassu}]
  In view of Remark \ref{remReg} we may assume  that  $\p_q^k h\in
  C^{2,\mu}(\bar R)$ for any $k\in\N$.
Now we prove the validity of  $(E_m)$ by using induction on $m$.   For $m=2
  $, $(E_m)$ obviously holds if we choose
\begin{eqnarray*}
L\geq \norm{\p_q^2 h}_{2,\mu}+1.
\end{eqnarray*}
 Now let $m\geq 3$ and  assume that $(E_j)$ holds for all $j\in\N$ with $2\leq j\leq
  m-1$, that is,
\begin{equation}\label{nomind}
\norm{\p^j_q h}_{2,\mu}\leq L^{j-1}(j-2)!, \quad 2\leq
j\leq m-1.
\end{equation}
Then we   show the validity of $(E_m)$.   For this purpose,
taking  the derivative with respect to $q$ up to order  $m$   on
both sides of  equations \reff{Equh1}-\reff{Equh3},  and then applying
Leibniz formula, we have
\begin{eqnarray}\label{oprtAB}
\left\{
\begin{array}{lll}
  A(h)[\p_q^m h]=f_1+f_2 \quad & {\rm in}~~R,\\[3pt]
  B(h)[\p_q^m h]=\varphi_1+\varphi_2\quad  &{\rm on}~~p=0,\\[3pt]
  \p_q^m h=0 \quad & {\rm on}~~p=p_0,
  \end{array}
  \right.
  \end{eqnarray}
where the operators $A(h)$ and $B(h)$ are defined by
\[
  A(h)[\phi]= (1+h_q^2)\phi_{pp}-2h_q h_p \phi_{qp}+h_p^2\phi_{qq},
  \quad \quad B(h)[\phi]=h_q\phi_q+(2gh-Q)h_p\phi_p+2gh_p^2\phi,
\]
and the right-hand side
\begin{eqnarray}
  f_1&=&\sum_{n=1}^m {m\choose n}\Big[-(\p_q^n h^2_q)(\p_q^{m-n}h_{pp})+2\inner{\p_q^n (h_ph_q)}(\p_q^{m-n} h_{pq})-(\p_q^n h^2_p)(\p_q^{m-n}h_{qq})\Big],\label{def f1}\\
  f_2&=&-\gamma(p)(\p_q^m h^3_p),\label{def f2}\\
  \varphi_1&=&-\frac{1}{2}\sum_{n=1}^{m-1}{m\choose n}(\p_q^n h_q)(\p_q^{m-n} h_q)-\frac{1}{2}(2gh-Q)\sum_{n=1}^{ m-1}{m\choose n}(\p_q^n h_p)(\p_q^{m-n} h_p),\label{defvp1}\\
  \varphi_2&=&-g\sum_{1\leq n\leq m-1}{m\choose n}(\p_q^n h)(\p_q^{m-n}
  h^2_p) \label{defvp2}.
\end{eqnarray}
The operator $A(h)$ is uniformly elliptic since its coefficients satisfy
\begin{eqnarray*}
  (1+h_q^2)h_p^2-h^2_q h^2_p=h_p^2\geq \inf_{\bar R}h_p^2>0
\end{eqnarray*}
due to \reff{hp bounded}. Also the boundary operator
$B(h)$ is uniformly oblique in the sense that it is bounded away from being tangential;
the coefficient $(2gh-Q)h_p$ of $\phi_p$ is nonzero and satisfies
\begin{eqnarray*}
  (2gh-Q)h_p=\frac{1+h_q^2}{h_p}\geq \frac{1}{\sup_{\bar R}h_p}\geq \delta
\end{eqnarray*}
in view of the boundary condition \reff{Equh2} and \reff{hp
  bounded}. Since $h\in C^{2,\mu}(\bar R)$  the coefficients of the operators $A(h)$ and $B(h)$ are in $C^{1,\mu}(\bar R)$. Moreover, by virtue of the induction assumption \reff{nomind}, one has $\p_q^i\p_p^j h\in C^{0,\mu}(\bar R)$ for all multi-index $(i,j)$ with $i+j\leq m+1$ and $j\leq 2$, and similarly $\p_q^i\p_p^j h\in C^{1,\mu}(\bar R)$ for all multi-index $(i,j)$ with $i+j\leq m$ and $j\leq 1$. As a result,
the right-hand side $f_i\in C^{0,\mu}(\bar R)$ and $\varphi_i\in
C^{1,\mu}(\bar R)$, $i=1,2$,   since by \reff{algebra} the
product of two functions in $C^{k,\mu}(\bar R)$ is still in
$C^{k,\mu}(\bar R)$ with $k=0,1$. Thus,  using the  standard Schauder
estimate (see for instance \cite[Theorem 6.30]{MR1814364}) we have,
\begin{eqnarray}\label{par m h}
  \norm{\p_q^m h}_{2,\mu}\leq \mathcal C \inner{\norm{\p_q^m h}_{0}+\sum_{i=1}^2\norm{f_i}_{0,\mu}+\sum_{i=1}^2\norm{\varphi_i}_{1,\mu}},
\end{eqnarray}
 where $\mathcal C$ is a constant depending only on $\mu,
\delta,\inf_{\bar R}h_p$ and $ \norm{h}_{2,\mu}$.  To show
 $(E_m)$ is valid, we estimate the terms on the  right-hand
 side of \reff{par m h} through the following
 steps.

To simplify the notations,  we will use $C_j, j\geq 1$, to denote
suitable {\it  harmless constants} larger than 1. By harmless
constants we mean these  are
independent of  $m$.

{\it Step 1)}~We claim that there exists $C_1>0$ such that, with $m\geq 3$,
\begin{equation}\label{step1}
\norm{\p_q^m h}_{0}\leq C_1 L^{m-2}(m-2)!.
\end{equation}
Indeed, when  $m=3$ the above estimate obviously holds if we choose
$C_1=\norm{h}_{3,\mu}+1$; when $m\geq 4$
 it follows  from the induction assumption \reff{nomind}  that
 \begin{eqnarray*}
   \norm{\p_q^m h}_{0}\leq \norm{\p_q^{m-2}
     h}_{2,\mu}\leq L^{m-3}(m-4)!\leq  L^{m-2}(m-2)!.
 \end{eqnarray*}
Then \reff{step1} follows.

{\it Step 2)}~Let $f_1$ be given in \reff{def f1}.  In this step we  prove
\begin{equation}\label{step2}
\norm{f_1}_{0,\mu}\leq  C_{2} L^{m-2}(m-2)!.
\end{equation}
Observe that ,   by \reff{algebra},
\begin{eqnarray}\label{est f1}
  \begin{split}
    \norm{f_1}_{0,\mu}& \leq& \sum_{ n=1}^{m}{m\choose
       n}\norm{\p_q^n h^2_q}_{0,\mu}\norm{\p_q^{m-n}h_{pp}}_{0,\mu}+2\sum_{ n=1}^{m} {m\choose
       n}\norm{\p_q^n (h_ph_q)}_{0,\mu} \norm{\p_q^{m-n} h_{pq}}_{0,\mu}\\
&&+\sum_{ n=1}^{m} {m\choose
       n}\norm{\p_q^n h^2_p}_{0,\mu}  \norm{\p_q^{m-n}h_{qq}}_{0,\mu}.
   \end{split}
\end{eqnarray}
We now treat the first term on the right-hand side, and write
   \begin{eqnarray}\label{est f1 1}
\begin{split}
     &\sum_{1\leq n\leq m} {m\choose
       n}\norm{\p_q^n
       h^2_q}_{0,\mu}\norm{\p_q^{m-n}h_{pp}}_{0,\mu}\leq
\sum_{1\leq n\leq m} {m\choose
       n}\norm{\p_q^n
       h^2_q}_{0,\mu}\norm{\p_q^{m-n}h}_{2,\mu}\\
&\leq \inner{\sum_{1\leq
         n\leq 2} +\sum_{3\leq n\leq m-2 } +\sum_{ m-1\leq n\leq m}
     }{m\choose n}\norm{\p_q^n
       h^2_q}_{0,\mu}\norm{\p_q^{m-n}h}_{2,\mu}.
\end{split}
 \end{eqnarray}
 By the induction assumption \reff{nomind}, one has
 \begin{eqnarray*}
  \forall~3\leq n \leq m, \quad \norm{\p_q^n h_q}_{0,\mu}\leq \norm{\p_q^{n-1} h}_{2,\mu}\leq L^{n-2}(n-3)!.
 \end{eqnarray*}
 Thus applying Lemma \ref{stab}, with $\ell=2$, $k_0=m$, $H=L$,
$u_1=u_2=h_q$ and $u_3=1$, yields that
 \begin{eqnarray}\label{par n hq2}
   \forall ~3\leq  n \leq m,\quad \norm{\p_q^n h^2_q}_{0,\mu}\leq C_5 L^{n-2}(n-3)!.
 \end{eqnarray}
 Moreover, we have
 \begin{eqnarray}\label{q m-n h}
  \forall~1\leq n\leq m-2,\quad \norm{\p_q^{m-n}h}_{2,\mu}\leq L^{m-n-1}(m-n-2)!
 \end{eqnarray}
due to the induction assumption \reff{nomind}.
Then using the above two estimates, straightforward verification shows that
\begin{equation}\label{1}
\sum_{n=1}^{2} {m\choose n}\norm{\p_q^n
  h^2_q}_{0,\mu}\norm{\p_q^{m-n}h}_{2,\mu}+\sum_{n=m-1}^{m} {m\choose n}\norm{\p_q^n
  h^2_q}_{0,\mu}\norm{\p_q^{m-n}h}_{2,\mu}
\leq C_{6}  L^{m-2}(m-3)!
\end{equation}
if we choose
\begin{eqnarray*}
 C_{6}\geq(\norm{h}_{3,\mu}+1)(30\norm{h}_{3,\mu}+4C_5+6).
\end{eqnarray*}
Next for the case when $3\leq n\leq m-2$, which  appears only when $m\geq 5$,
combination of the estimates \reff{par n hq2} and \reff{q m-n h} gives
\begin{eqnarray*}
\sum_{n=3}^{m- 2} {m\choose n}\norm{\p_q^n
  h^2_q}_{0,\mu}\norm{\p_q^{m-n}h}_{2,\mu}
&\leq & C_5 \sum_{3\leq n\leq m- 2} \frac{m!}{n!(m-n)!} L^{n-2}(n-3)!
L^{m-n-1}(m-n-2)!\\
&\leq &  C_{7}L^{m-3}(m-2)!\sum_{3\leq n\leq m- 2}\frac{m^2}{n^3(m-n)^2}\\
&\leq & C_{8}L^{m-3}(m-2)!.
\end{eqnarray*}
This along with  \reff{1} shows, in view of \reff{est f1 1},
\begin{eqnarray*}
  \sum_{1\leq n\leq m} {m\choose
       n}\norm{\p_q^n h^2_q}_{0,\mu}\norm{\p_q^{m-n}h_{pp}}_{0,\mu}
     \leq  \inner{C_{6}+C_{8}}L^{m-2}(m-2)!.
\end{eqnarray*}
Similarly,
we can find a constant $C_{9}$ such that
\begin{eqnarray*}
  2\sum_{n=1}^{m} {m\choose
       n}\norm{\p_q^n (h_ph_q)}_{0,\mu} \norm{\p_q^{m-n} h_{pq}}_{0,\mu}
 +\sum_{n=1}^{m} {m\choose
       n}\norm{\p_q^n h^2_p}_{0,\mu}  \norm{\p_q^{m-n}h_{qq}}_{0,\mu}\leq C_{9}L^{m-2}(m-2)!.
\end{eqnarray*}
Inserting the above two estimates into \reff{est f1}, we get the desired estimate
\reff{step2} by choosing $C_{2}=C_{6}+C_{8}+C_{9}$.

{\it Step 3)}~Let $f_2$ be given in \reff{def f2}.  We now prove
\begin{equation}\label{step3}
\norm{f_2}_{0,\mu}\leq  C_{3} L^{m-2}(m-2)!.
\end{equation}
 In fact,
using  \reff{algebra} we have
   \begin{eqnarray}\label{est f4}
   \begin{split}
     \norm{f_2}_{0,\mu}&\leq \norm{ \gamma}_{0,\mu} \norm{\p_q^m h^3_p}_{0,\mu}.
   \end{split}
 \end{eqnarray}
By the induction assumption \reff{nomind}, one has
\begin{eqnarray*}
\forall~3\leq j\leq m,\quad \norm{\p_q^j h_p}_{0,\mu}\leq \norm{\p_q^{j-1} h}_{2,\mu}\leq L^{j-2}(j-3)!.
\end{eqnarray*}
Then  using  Lemma \ref{stab},   with $\ell=2$,  $k_0=m$ ,  $H=L$,
$u_1=u_2=u_3=h_p$,  we conclude
\begin{eqnarray*}
  \norm{\p_q^m h_p^3}_{0,\mu}\leq C_{10}  L^{m-2}(m-3)!.
\end{eqnarray*}
Choosing $C_{3}=C_{10}\norm{\gamma}_{0,\mu}+1$, we obtain \reff{step3} in view of \reff{est f4}.

{\it Step 4)}~ Finally  we  prove, with $\varphi_1$ and $\varphi_2$ given in \reff{defvp1} and \reff{defvp2},
\begin{equation}\label{step4}
  \sum_{i=1}^2\norm{\varphi_i}_{1,\mu}\leq C_{4}L^{m-2}(m-2)!.
\end{equation}
First for $\norm{\varphi_1}_{1,\mu}$, we have
\begin{eqnarray*}
  \norm{\varphi_1}_{1,\mu}\leq C_{11}L^{m-2}(m-2)!.
\end{eqnarray*}
The proof is quite similar as that of \reff{step2} for $\norm{f_1}_{0,\mu}$, and is in fact simpler since we do not need to use Lemma \ref{stab}, so we omit the details.
 Next for $\norm{\varphi_2}_{1,\mu}$,  we write,  by \reff{algebra},
 \begin{eqnarray}\label{est varphi2}
   \begin{split}
   \norm{\varphi_2}_{1,\mu}&\leq 2g\sum_{1\leq n\leq m-1}{m\choose n}\norm{\p_q^n h}_{1,\mu}\norm{\p_q^{m-n} h^2_p}_{1,\mu}\\
   &\leq 2g\inner{\sum_{1\leq n\leq 2}+\sum_{3\leq n\leq m-2}+\sum_{n= m-1}}{m\choose n}\norm{\p_q^{n} h}_{1,\mu}\norm{\p_q^{m-n} h^2_p}_{1,\mu}.
   \end{split}
 \end{eqnarray}
 By the induction assumption \reff{nomind}, one has
 \begin{equation}\label{3}
  \forall~3\leq n\leq m,\quad \norm{\p_q^{n} h}_{1,\mu}\leq\norm{\p_q^{n-1} h}_{2,\mu}\leq L^{n-2}(n-3)!,
 \end{equation}
 and for $1\leq n\leq m-2$,
\begin{eqnarray*}
\forall~2\leq j\leq m-n,\quad \norm{\p_q^j h_p}_{1,\mu}\leq \norm{\p_q^{j} h}_{2,\mu}\leq L^{j-1}(j-2)!.
\end{eqnarray*}
This last estimate allows us to use Lemma \ref{stab},   with $\ell=1,$
$k_0=m-n$ with $1\leq n\leq m-2$ ,  $H=L$,
$u_1=u_2=h_p$ and $u_3=1$,  to conclude,
\begin{equation}\label{4}
 \forall~1\leq n\leq m-2,\quad \norm{\p_q^{m-n} h^2_p}_{1,\mu}\leq C_{12} L^{m-n-1}(m-n-2)!.
\end{equation}
In virtue of \reff{3} and \reff{4}, direct verification shows
\begin{eqnarray}\label{est varphi2 1}
   2g\inner{\sum_{1\leq n\leq 2}+\sum_{n= m-1}}{m\choose n}\norm{\p_q^{n} h}_{1,\mu}\norm{\p_q^{m-n} h^2_p}_{1,\mu}\leq C_{13}L^{m-2}(m-2)!.
\end{eqnarray}
Next for the case when $3\leq n\leq m-2$, which appears only when
$m\geq 5$, we use again \reff{3} and \reff{4} to compute
\begin{eqnarray*}
    &&2g\sum_{3\leq n\leq m-2}{m\choose n}\norm{\p_q^{n} h}_{1,\mu}\norm{\p_q^{m-n} h^2_p}_{1,\mu}\\
   & \leq &2gC_{12}\sum_{3\leq n\leq m-2}\frac{m!}{n!(m-n)!}L^{n-2}(n-3)!L^{m-n-1}(m-n-2)!\\
   &\leq & C_{14}L^{m-3}(m-2)!\sum_{3\leq n\leq m-2}\frac{m^2}{n^3(m-n)^2}\\
   &\leq & C_{15}L^{m-3}(m-2)!.
\end{eqnarray*}
Inserting \reff{est varphi2 1}  and the  above estimate  into \reff{est varphi2},
we obtain
\[
    \norm{\varphi_2}_{1,\mu}\leq \inner{C_{13}+C_{15}}L^{m-2}(m-2)!.
\]
Thus the desired estimate \reff{step4} follows by choosing $C_4=C_{11}+C_{13}+C_{15}$.

Now we come back to the proof of Proposition \ref{propassu}.  Choose
$L$ in such a way that
\begin{equation*}
 L\geq\mathcal C\inner{C_1+ C_2+ C_3+ C_4}+\norm{\p_q^2 h}_{2,\mu}+1
\end{equation*}
with $\mathcal C, C_1,\cdots,C_4$ the constants given in \reff{par m h},
\reff{step1}, \reff{step2}, \reff{step3} and \reff{step4}.   Then
combining
\reff{par m h},  \reff{step1}, \reff{step2}, \reff{step3} and \reff{step4},  we have,
 \begin{eqnarray*}
    \norm{\p_q^m h}_{2,\mu}\leq \mathcal C\inner{C_1+ C_2+ C_3+ C_4} L^{m-2}(m-2)!\leq L^{m-1}(m-2)!.
 \end{eqnarray*}
 The validity of $(E_m)$ follows. Thus  the proof of
 Proposition \ref{propassu} is complete.
\end{proof}

The rest of this section is occupied by

\begin{proof}[Proof of Lemma \ref{stab}]
In what follows we always assume $ \ell+1\leq k\leq k_0$.    To simplify the notation, we use $b_i, i\geq 1$, to denote suitable
constants larger than $1$, depending  only on $\ell$.

Firstly let $u_3\equiv 1$.
By Leibniz formula we have
\begin{eqnarray*}
  \p_q^k\inner{u_1u_2}=\sum_{0\leq j\leq k} {k \choose j}
   \inner{ \p_q^ju_1}   \inner{ \p_q^{k-j}u_2}.
\end{eqnarray*}
Note that $\norm{\cdot}$ stands for the
H\"{o}lder norm $\norm{\cdot}_{0,\mu}$ or  $\norm{\cdot}_{1,\mu}$.
Then  from \reff{algebra} it follows that
\begin{eqnarray*}
  \norm{\p_q^k\inner{u_1u_2}}&\leq& 2 \sum_{0\leq j\leq k} \frac{k!}{j !(k-j)!}
   \norm{ \p_q^ju_1} \norm{
       \p_q^{k-j}u_2}\\
&\leq& S_1+S_2+S_3
\end{eqnarray*}
with
\begin{eqnarray*}
       S_1&=&2 \sum_{{0\leq j\leq \ell}}  \frac{k!}{j !(k-j)!}
   \norm{\p_q^ju_1} \norm{
       \p_q^{k-j}u_2},\\
   S_2&=&2\sum_{{\ell+1\leq j\leq k-\ell-1}}  \frac{k!}{j !(k-j)!}
   \norm{ \p_q^ju_1} \norm{
       \p_q^{k-j}u_2},\\
S_3&=&2 \sum_{{k-\ell\leq j\leq k}}  \frac{k!}{j !(k-j)!}
   \norm{ \p_q^ju_1} \norm{
       \p_q^{k-j}u_2}.
\end{eqnarray*}
Using the assumption \reff{condition3}, direct computation shows that
there exists a constant $b_1>1$, depending only on $\ell$, such that
   \begin{eqnarray*}
    S_1+S_3\leq b_1\inner{\norm{u_1}_{\ell+1,\mu}+\norm{u_2}_{\ell+1,\mu}+1}^2 H^{k-\ell}(k-\ell-1)!.
   \end{eqnarray*}
For $S_2$, which appears only when $k\geq 2\ell+2$, we have
\begin{eqnarray*}
    S_2&\leq& 2 \sum_{{\ell+1\leq j\leq k-\ell-1}}\frac{k!}{j !(k-j)!}
     H^{j-\ell}(j-\ell-1)!H^{k-j-\ell}(k- j-\ell-1)! \\
  &\leq& b_2\sum_{{\ell+1\leq j\leq k-\ell-1}}\frac{k!}{j^{\ell+1} (k-j)^{\ell+1}}
     H^{k-2\ell}\\
  &\leq& b_3 H^{k-2\ell} (k-\ell-1)! \sum_{{\ell+1\leq j\leq k-\ell-1}} \frac{k^{\ell+1}}{ j^{\ell+1}\inner{k-j}^{\ell+1}}\\
  &\leq&  b_4 H^{k-\ell}(k-\ell-1)!.
\end{eqnarray*}
In view of the estimates for $S_1,S_2$ and $S_3$, we conclude
\begin{equation*}
  \norm{\p_q^{k}(u_1u_2)}\leq b_5 \inner{\norm{u_1}_{\ell+1,\mu}+\norm{u_2}_{\ell+1,\mu}+1}^2  H^{k-\ell}(k-\ell-1)!
\end{equation*}
by choosing $b_5=b_1+b_4$.

Now we consider the case when $u_3\not\equiv 1$.  We have shown above
that
\begin{eqnarray*}
  \forall~\ell+1\leq k\leq k_0,  \quad  \|\p_q^k (u_1u_2)\|
  \leq b_5 \inner{\norm{u_1}_{\ell+1,\mu}+\norm{u_2}_{\ell+1,\mu}+1}^2  H^{k-\ell}(k-\ell-1)!,
\end{eqnarray*}
provided $u_1$ and  $u_2$ satisfy \reff{condition3}.
This allows us to use the same argument as above to the two
functions $$b_5^{-1}\inner{\norm{u_1}_{\ell+1,\mu}+\norm{u_2}_{\ell+1,\mu}+1}^{-2}
u_1u_2 ~{\rm and}~ u_3; $$
this gives, for any $\ell+1\leq k\leq k_0$,
\begin{eqnarray*}
   \norm{\p_q^k\inner{u_1u_2 u_3}}&\leq&  b_5^2
   \inner{\norm{u_1}_{\ell+1,\mu}+\norm{u_2}_{\ell+1,\mu}+1}^2\inner{\norm{u_1u_2}_{\ell+1,\mu}+\norm{u_3}_{\ell+1,\mu}+1}^2
   H^{k-\ell}(k-\ell-1)!\\
&\leq&  b_6
   \inner{\norm{u_1}_{\ell+1,\mu}+\norm{u_2}_{\ell+1,\mu}+\norm{u_3}_{\ell+1,\mu}+1}^6 H^{k-\ell}(k-\ell-1)!.
\end{eqnarray*}
The conclusion follows by choosing $C_*=b_6$. Then
the proof of Lemma \ref{stab} is complete.
\end{proof}

\section{Gevrey regularity of stream function}\label{sec4}

Let $G^s\inner{[p_0,  0]}$,  $s\geq 1$,  be the Gevrey class; see
Definition \ref{def gevrey} of Gevrey function.   In
this section we assume  $\gamma\in G^s\inner{[p_0,  0]}$.  Then by
the alternative characterization  of Gevrey function,   for any $p\in[p_0, 0]$ we can
find a neighborhood $U_p$ of $p$ and a constant $M_p$ such that
\[
\forall~k\in \mathbb N,\quad \sup_{t\in U_p\cap [p_0,0]} \abs{\p_p^k
  \gamma(t)}\leq M_p^{k+1}(k!)^s.
\]
Note $[p_0,0]$ is compact in $\mathbb R$; this allows us to find a constant $M$ such
that
\begin{equation}\label{GevCon}
\forall~k\in \mathbb N,\quad \sup_{p\in[p_0, 0]}\abs{\p_p^k
  \gamma(p)}\leq M^{k+1}(k!)^s.
\end{equation}

We prove now the Gevrey regularity of stream function,
i.e., Theorem \ref{th2}. In view of Proposition \ref{reserv},  it suffices to
show the following result for the height function $h(q,p)$.

\begin{proposition}\label{prophpq}
 Let $\gamma\in G^s\inner{[p_0,  0]}$ with $s\geq 1$,  and let $h\in C^{2,\mu}(\bar
 R)$ be a solution to \reff{Equh1}-\reff{Equh3}.
Then there exist two  constants $ L_1, L_2$ with $L_2\geq L_1\geq 1$ ,  such
that for any $m\geq 2$  we have the
following estimate
   \begin{eqnarray*}
     (F_m):\qquad \forall~\alpha=(\alpha_1,\alpha_2)\in\N^2,~ \abs\alpha=m,\quad  \norm{\p^{\alpha}h}_{2}\leq
     L_1^{\alpha_1-1}L_2^{\alpha_2} [(\abs\alpha-2)!]^s.
   \end{eqnarray*}
Recall $\norm{\cdot}_2$ stands for the  H\"{o}lder norm $\norm{\cdot}_{C^{2,0}(\bar R)}$. Thus $h\in
 G^s(\bar R)$;  in particular if $s=1$ then $h$ is analytic
 in $\bar R$.
\end{proposition}

\begin{remark}
 As to be seen in the proof, the constants $L_1, L_2$ depend on  the
 constant $L$  given in Proposition \ref{propassu} and the constant $M$
 in \reff{GevCon},  but
 independent of the order $m$ of derivative.
\end{remark}

\begin{remark}\label{+remReg}
Note   $\gamma\in G^s([p_0,0])\subset C^\infty([p_0,0])$.  By  Remark
\ref{remReg}   we see $\p_q h \in C^{2,\mu}(\bar R)$.  Then
differentiating the
   equation \reff{Equh1} with respect to $p$,    we can  obtain $h\in
   C^{3, \mu}(\bar R)$;  see
   \cite{MR2027299} for details.    Repeating  this procedure gives  $ h\in
  C^{k,\mu}(\bar R)$  for any $k \in \mathbb N$,  since $\gamma\in C^\infty([p_0,0])$.
\end{remark}

To confirm the last statement in the above Proposition \ref{prophpq},
we   choose $C$ in such a way
that
\[
     C=\max\set{L_1, L_2,  \norm{h}_{1,\mu}},
\]
which,  along with the estimate $(F_m)$ with $m\geq 2$   in Proposition \ref{prophpq},  yields
\begin{eqnarray*}
\forall~\alpha\in\mathbb N^2, \quad \max_{(q,p)\in
  \bar R}|\partial^\alpha{h}(q,p)| \leq C^{\abs\alpha+1}(\abs \alpha!)^s .
\end{eqnarray*}
This gives $h\in
 G^s(\bar R)$.

In order to prove the above proposition, we need the following
technical lemma.

\begin{lemma}\label{+stab}

Let $s\geq 1$, and $H_1$ and $H_2$ be two  constants with $H_2\geq
H_1\geq 1$. Suppose that $\alpha_0$ is a given multi-index with $\abs{\alpha_0}\geq3$, and $u,v,w\in
C^{\abs{\alpha_0},\mu}(\bar R)$.  For $j=0,1,2$, denote
\begin{equation*}
 \mathcal{A}_{j}=\Big\{f\in C^{|\alpha_0|,\mu}(\bar R)~\big|~
   \forall~\alpha=(\alpha_1,\alpha_2)\leq\alpha_0, \,\abs{\alpha}\geq j+1, ~  \|\p^\alpha f\|_0
  \leq H_1^{\alpha_1-j}H_2^{\alpha_2}[(\abs\alpha-j-1)!]^s\Big\}.
\end{equation*}
Then there exists a constant $c_*$, depending only on the $C^{2,0}$-norms
of $u, v$ and $w$, but independent of $\alpha_0$,  such that
\begin{enumerate}[(a)]
  \item\label{conc1} if $u\in \mathcal{A}_2$ and $v\in \mathcal{A}_1$,
    then $c_*^{-1}uv\in \mathcal{A}_1$, that is,
  \[
   \forall~\alpha=(\alpha_1,\alpha_2)\leq\alpha_0,~ \abs{\alpha}\geq 2,  \quad \big\|\p^\alpha\inner{uv}\big\|_0\leq c_* H_1^{\alpha_1-1}H_2^{\alpha_2}[(\abs\alpha-2)!]^s;
  \]
 if additionally  $w\in \mathcal A_2$ then $c_*^{-1}uvw\in
 \mathcal{A}_1$;
  \item\label{conc2}  if $u\in \mathcal{A}_2$ and $v,w\in \mathcal{A}_1$, then $c_*^{-1}uvw\in \mathcal{A}_0$;
  \item\label{conc3}  if $u,v\in \mathcal{A}_2$ and $w\in
    \mathcal{A}_0$, then $c_*^{-1}uvw\in \mathcal{A}_0$;
 \item\label{conc4}  if $u, v, w\in \mathcal{A}_2$, then $c_*^{-1}H_1
   uvw\in \mathcal{A}_1$, that is,
\[
   \forall~\alpha=(\alpha_1,\alpha_2)\leq\alpha_0,~ \abs{\alpha}\geq 2,  \quad \big\|\p^\alpha\inner{uvw}\big\|_0\leq c_* H_1^{\alpha_1-2}H_2^{\alpha_2}[(\abs\alpha-2)!]^s.
  \]
\end{enumerate}

\end{lemma}

The proof of the above lemma is postponed to the end of this section. Now
we prove our main result.

   \begin{proof}[Proof of Proposition \ref{prophpq}]
In view of Remark \ref{+remReg} we may assume  that  $ h\in
  C^{k,\mu}(\bar R)$ for any $k\in\N$.  We now use induction on $m$ to
  prove the estimate $(F_m)$.  First for $m=2$,
   $(F_m)$ obviously holds by  choosing  $ L_1, L_2$ in such a way
   that
 \begin{equation}\label{l1l21}
   L_2\geq L_1\geq \norm{h}_{4}+1.
\end{equation}
Next let $m\geq 3$  and assume that
   $(F_j)$ holds for any $j$ with $2\leq j\leq m-1$, that is,
\begin{equation}\label{fm1}
\forall~\beta=(\beta_1,\beta_2),~ 2\leq
\beta_1+\beta_2\leq m-1, \quad \norm{\p^{\beta}h}_{2}\leq L_1^{\beta_1-1}
L_1^{\beta_2}[(\abs\beta-2)!]^s.
\end{equation}
We have to prove the validity of $(F_m)$.  This is equivalent to show
the following estimate
\begin{equation}\label{fmn}
(F_{m,n}): \qquad\norm{\p_q^{m-n}\p_p^{n}h}_{2}\leq L_1^{m-n-1}
L_2^{n}[(m-2)!]^s
\end{equation}
holds for all $n$ with $0\leq n\leq m.$

In what follows we use induction on $n$ to show \reff{fmn} with fixed
$m\geq 3$.
Firstly note that $s\geq 1$, and thus from Proposition \ref{propassu} we see that $(F_{m,0})$ holds
   if we choose
\begin{equation}\label{l1l22}
L_1\geq L.
\end{equation}   Next let $1\leq n\leq m$ and  assume that $(F_{m,i})$
holds for all $i$ with $0\leq
   i\leq n-1$, that is,
\begin{equation}\label{fn1}
\forall~0\leq i\leq n-1,\quad \norm{\p_q^{m-i}\p_p^i h}_{2}\leq L_1^{m-i-1}
L_2^{i}[(m-2)!]^s.
\end{equation}
We have to show $(F_{m,n})$ holds as well, i.e., to prove that
   \begin{equation}\label{hqmn}
     \norm{\p_q^{m-n}\p_p^n h}_{2}\leq L_1^{m-n-1}L_2^{n}[(m-2)!]^s.
   \end{equation}
To do so, we firstly compute, with $1\leq n\leq m$,
   \begin{eqnarray*}
     \norm{\p_q^{m-n}\p_p^n h}_{2}&\leq&\norm{\p_q^{m-n}\p_p^n h}_{1}+\norm{\p_q^{m-n+2}\p_p^{n} h}_{0}+\norm{\p_q^{m-n+1}\p_p^{n+1} h}_{0}+\norm{\p_q^{m-n}\p_p^{n+2} h}_{0}\\&\leq&\norm{\p_q^{m-n}\p_p^{n-1} h}_{2}+2\norm{\p_q^{m-(n-1)}\p_p^{n-1} h}_{2}+\norm{\p_q^{m-n}\p_p^{n+2} h}_{0}.
   \end{eqnarray*}
  The induction assumptions \reff{fm1} and \reff{fn1} yield
   \begin{eqnarray*}
     \norm{\p_q^{m-n}\p_p^{n-1} h}_{2}+2\norm{\p_q^{m-(n-1)}\p_p^{n-1} h}_{2}
       &\leq& L_1^{m-n-1}L_2^{n-1}[(m-3)!]^s+2L_1^{m-n}L_2^{n-1}[(m-2)!]^s\\
       &\leq& L_2^{-1}(1+2L_1) L_1^{m-n-1}L_2^{n}[(m-2)!]^s\\
       &\leq& \frac{1}{2}L_1^{m-n-1}L_2^{n}[(m-2)!]^s,
   \end{eqnarray*}
   where in the last inequality we choose
   \begin{equation}\label{l1l23}
       L_2\geq 8L_1\geq 8.
\end{equation}
 Accordingly, in order to obtain
   \reff{hqmn}, it  suffices to prove
\begin{equation}\label{+hqmn}
     \norm{\p_q^{m-n}\p_p^{n+2} h}_{0}\leq  \frac{1}{2} L_1^{m-n-1}L_2^{n}[(m-2)!]^s.
   \end{equation}
The rest is occupied by the proof of the above estimate.

From now on we fix $m$  and $n$ with $m\geq 3$ and $1\leq n\leq m$,  and denote $\alpha=(\alpha_1,\alpha_2)=(m-n,n)$.  Applying $\p^{\alpha}=\p_q^{m-n}\p_p^n$ on both sides of the equation \reff{Equh1} gives
   \begin{eqnarray*}
     (1+h_q^2)(\p^\alpha h_{pp})=-\sum_{\beta\leq \alpha,\beta\neq 0}{\alpha\choose\beta}(\p^\beta h_q^2)(\p^{\alpha-\beta}h_{pp})+2\p^{\alpha}(h_p h_q h_{qp})-\p^{\alpha}(h_p^2h_{qq})-\p^{\alpha}\inner{\gamma h_p^3},
   \end{eqnarray*}
   which implies
   \begin{eqnarray*}
     \norm{(1+h_q^2)(\p^\alpha h_{pp})}_{0}&\leq& \sum_{\beta\leq \alpha,\beta\neq 0}{\alpha\choose\beta}\norm{\p^\beta h_q^2}_{0}\norm{\p^{\alpha-\beta}h_{pp}}_{0}+2\norm{\p^{\alpha}(h_p h_q h_{qp})}_{0}\\[3pt]
     &&+\norm{\p^{\alpha}(h_p^2h_{qq})}_{0}+\norm{\p^{\alpha}\inner{\gamma h_p^3}}_{0}.
   \end{eqnarray*}
   Since
   \begin{equation}\label{9}
     \norm{\p_q^{m-n}\p_p^{n+2} h}_{0}=\norm{\p^\alpha h_{pp}}_{0}\leq\norm{(1+h_q^2)(\p^\alpha h_{pp})}_{0},
   \end{equation}
we obtain,  with  $\alpha=(\alpha_1,\alpha_2)=(m-n,n)$,
\begin{equation}\label{5}
    \begin{split}
     \norm{\p_q^{m-n}\p_p^{n+2} h}_{0}\leq& \sum_{\beta\leq \alpha,\beta\neq 0}{\alpha\choose\beta}\norm{\p^\beta h_q^2}_{0}\norm{\p^{\alpha-\beta}h_{pp}}_{0}+2\norm{\p^{\alpha}(h_p h_q h_{qp})}_{0}\\
     &+\norm{\p^{\alpha}(h_p^2h_{qq})}_{0}+\norm{\p^{\alpha}\inner{\gamma h_p^3}}_{0}.
    \end{split}
   \end{equation}
We now treat the  terms on the right-hand side through the
following lemmas.

To simplify the notations,  we will use $c_j, j\geq 1$, to denote
suitable {\it harmless constants} larger than $1$.  By harmless constants it
means that these constants are independent of
$m$ and $n$.

    \begin{lemma}\label{lem401}
   For $\alpha=(\alpha_1,\alpha_2)=(m-n,n)$ with
     $1\leq n\leq m$,  we have
      \begin{equation*}
       \sum_{\beta\leq \alpha,\beta\neq
        0}{\alpha\choose\beta}\norm{\p^\beta
        h_q^2}_{0}\norm{\p^{\alpha-\beta}h_{pp}}_{0}\leq c_1 L_1^{\alpha_1-2}L_2^{\alpha_2} [(\abs{\alpha}-2)!]^{s}.
    \end{equation*}
    \end{lemma}

    \begin{proof}[Proof of the lemma]
      We firstly use Lemma \ref{+stab} to  treat the term
      $\norm{\p^\beta h_q^2}_0$ with $3\leq \abs\beta\leq
      \abs\alpha=m$.   To do so, write
      $\beta=\tilde\beta+(\beta-\tilde\beta)$ with
      $|\tilde\beta|=\abs\beta-1\geq 2$.  Without loss of  generality we may
      take $\beta-\tilde\beta=(0,1)$, and the arguments below also
      holds when $\beta-\tilde\beta=(1,0)$.   Thus
  \begin{equation}\label{rel}
  \p^\beta h_q^2=2\p^{\tilde \beta} \inner{h_q h_{qp}}.
\end{equation}
Note that for any
      $\xi=(\xi_1,\xi_2)\leq \tilde\beta$ with
      $|\xi|\geq 3$, we have, using the induction assumption
      \reff{fm1},
\[
  \norm{\p^{\xi} h_{qp}}_{0}\leq \norm{\p^{\xi} h}_{2}\leq L_1^{\xi_1-1}L_2^{\xi_2}
[(|\xi|-2)!]^s,
\]
and
    \begin{equation}\label{10}
      \norm{\p^{\xi} h_q}_{0}\leq
\left\{
\begin{array}{lll}
\norm{\p_q^{\xi_1}\p_p^{\xi_2-1} h}_{2} \leq
L_1^{\xi_1-1} L_2^{\xi_2-1}
[(|\xi|-3)!]^s\leq L_1^{\xi_1-2}L_2^{\xi_2}
[(|\xi|-3)!]^s,\quad &\xi_2\geq 1,\\[3pt]
\norm{\p_q^{\xi_1-1}h}_{2} \leq
L_1^{\xi_1-2}
[(|\xi|-3)!]^s= L_1^{\xi_1-2}L_2^{\xi_2}
[(|\xi|-3)!]^s,\quad &\xi_2=0,
\end{array}
    \right.
    \end{equation}
    where in the case $\xi_2\geq 1$ we used $L_2\geq L_1$.
 Therefore applying Lemma \ref{+stab}-\reff{conc1}, with  $H_1=L_1$, $H_2=L_2$,  $u=h_q$ and
 $v=h_{qp}$,  gives
\[
  \norm{\p^{\tilde \beta} (h_qh_{qp})}_{0}\leq c_4
      L_1^{\tilde \beta_1-1}L_2^{\tilde \beta_2} [(|\tilde
      \beta|-2)!]^s=  c_4
      L_1^{ \beta_1-1}L_2^{\beta_2-1} [(|
      \beta|-3)!]^s\leq  c_4
      L_1^{ \beta_1-2}L_2^{\beta_2} [(|
      \beta|-3)!]^s,
\]
the last inequality holding because $L_2\geq
L_1$. This along with the relation \reff{rel} yields
    \begin{equation}\label{6}
      \forall ~\beta, ~3\leq \abs{\beta}\leq
      \abs\alpha,\quad\norm{\p^{\beta} h^2_q}_{0}\leq 2 c_4
      L_1^{\beta_1-2}L_2^{\beta_2} [(|\beta|-3)!]^s.
    \end{equation}
    On the other hand, for the term
    $\norm{\p^{\alpha-\beta}h_{pp}}_{0}$, we have,  by the induction
  assumption \reff{fm1},
    \begin{equation}\label{7}
      \forall~\beta\leq \alpha,~1\leq\abs{\beta}\leq\abs{\alpha}-2,\quad \norm{\p^{\alpha-\beta}h_{pp}}_{0}\leq
      \norm{\p^{\alpha-\beta}h}_{2}\leq
      L_1^{\alpha_1-\beta_1-1}L_2^{\alpha_2-\beta_2}[(\abs{\alpha}-\abs{\beta}-2)!]^s.
    \end{equation}

    Next we write
    \begin{equation*}
     \begin{split}
      \sum_{\beta\leq \alpha,\beta\neq 0}{\alpha\choose\beta}\norm{\p^\beta h_q^2}_{0}\norm{\p^{\alpha-\beta}h_{pp}}_{0}&=\inner{\sum_{\stackrel{\beta\leq \alpha}{1\leq\abs{\beta}\leq 2}}+\sum_{\stackrel{\beta\leq \alpha}{3\leq\abs{\beta}\leq \abs{\alpha}-2}}+\sum_{\stackrel{\beta\leq \alpha}{\abs{\beta}\geq \abs{\alpha}-1}}}{\alpha\choose\beta}\norm{\p^\beta h_q^2}_{0}\norm{\p^{\alpha-\beta}h_{pp}}_{0}\\
      &=J_1+J_2+J_3.
      \end{split}
    \end{equation*}
    By virtue of \reff{6} and \reff{7},  direct computation as in
    \reff{1},  shows that
    \begin{equation*}
      J_1+J_3\leq c_{5}L_1^{\alpha_1-2}L_2^{\alpha_2}[(\abs{\alpha}-2)!]^s.
    \end{equation*}
 Next for $J_2$, which appears only when $\abs{\alpha}\geq 5$, we have by \reff{6} and \reff{7} that
    \begin{equation*}
      \begin{split}
        J_2&\leq 2c_4\sum_{\stackrel{\beta\leq
            \alpha}{3\leq\abs{\beta}\leq
            \abs{\alpha}-2}}\frac{\abs{\alpha}!}{\abs{\beta}!(\abs{\alpha}-\abs{\beta})!}
             L_1^{\beta_1-2}L_2^{\beta_2} [(|\beta|-3)!]^sL_1^{\alpha_1-\beta_1-1}L_2^{\alpha_2-\beta_2}[(\abs{\alpha}-\abs{\beta}-2)!]^s\\
      &\leq c_{6}\sum_{\stackrel{\beta\leq \alpha}{3\leq\abs{\beta}\leq \abs{\alpha}-2}}\frac{\abs{\alpha}!}{\abs{\beta}^3(\abs{\alpha}-\abs{\beta})^2}
     L_1^{\alpha_1-3}L_2^{\alpha_2} [(|\beta|-3)!]^{s-1}[(\abs{\alpha}-\abs{\beta}-2)!]^{s-1}\\
     &\leq  c_{6}L_1^{\alpha_1-3}L_2^{\alpha_2}\sum_{\stackrel{\beta\leq \alpha}{3\leq\abs{\beta}\leq \abs{\alpha}-2}}\frac{\abs{\alpha}!}{\abs{\beta}^3(\abs{\alpha}-\abs{\beta})^2}
      [(\abs{\alpha}-5)!]^{s-1}\\
      &\leq c_{6}L_1^{\alpha_1-3}L_2^{\alpha_2} [(\abs{\alpha}-2)!]^{s}\sum_{\stackrel{\beta\leq \alpha}{3\leq\abs{\beta}\leq \abs{\alpha}-2}}\frac{\abs{\alpha}^2}{\abs{\beta}^3(\abs{\alpha}-\abs{\beta})^2}\\
      &\leq c_{7}L_1^{\alpha_1-3}L_2^{\alpha_2} [(\abs{\alpha}-2)!]^{s},
      \end{split}
    \end{equation*}
    the last inequality holding because 
\[
\sum_{\stackrel{\beta\leq \alpha}{3\leq\abs{\beta}\leq \abs{\alpha}-2}}\frac{\abs\alpha^2}{\abs\beta^3(\abs\alpha-\abs\beta)^2}
\leq
8\pi^2.
\]   Therefore, choosing $c_1=c_5+c_7$,   we can
    combine the estimates for $J_1, J_2$ and $ J_3$ to complete the proof of the lemma.
    \end{proof}

    \begin{lemma}
  For $\alpha=(\alpha_1,\alpha_2)=(m-n,n)$ with
     $1\leq n\leq m$, we have
      \begin{equation*}
        2\norm{\p^{\alpha}(h_p h_q
          h_{qp})}_{0}+\norm{\p^{\alpha}(h_p^2h_{qq})}_{0}\leq c_2 L_1^{\alpha_1}L_2^{\alpha_2-1} [(\abs{\alpha}-2)!]^{s}.
      \end{equation*}
    \end{lemma}

    \begin{proof}[Proof of the lemma]
    Since $n\geq 1$, we can write $\alpha=\tilde\alpha+(0,1)$ with $\tilde \alpha=(\tilde\alpha_1,\tilde\alpha_2)=(m-n,n-1)$.
    Thus
    \begin{equation}\label{d}
      \p^{\alpha}(h_p h_q h_{qp})=\p^{\tilde\alpha}\inner{h_{pp} h_q h_{qp}+h_p h_{qp} h_{qp}+h_p h_q h_{qpp}}.
    \end{equation}
    We next compute the estimate for the term $\p^{\tilde\alpha}(h_{pp} h_q h_{qp})$.
      For any $\beta\leq \tilde\alpha$ with $\abs{\beta}\geq 3$,  we
      have, as for $\norm{\p^{\xi} h_q}_{0}$ in \reff{10},
      \begin{equation*}
        \norm{\p^\beta h_q}_0\leq L_1^{\beta_1-2}L_2^{\beta_2}[(\abs{\beta}-3)!]^s,\qquad \norm{\p^\beta h_p}_0\leq L_1^{\beta_1-2}L_2^{\beta_2}[(\abs{\beta}-3)!]^s,
      \end{equation*}
      and by  the induction assumption  \reff{fm1} and \reff{fn1}, in view of
      $\beta_2\leq \tilde\alpha_2= n-1$,
\begin{eqnarray*}
\norm{\p^\beta h_{pp}}_0&\leq&\norm{\p^\beta h}_2\leq
L_1^{\beta_1-1}L_2^{\beta_2}[(\abs{\beta}-2)!]^s,\\
\norm{\p^\beta h_{qp}}_0&\leq &\norm{\p^\beta h}_2\leq
L_1^{\beta_1-1}L_2^{\beta_2}[(\abs{\beta}-2)!]^s,\\
\norm{\p^\beta h_{qpp}}_0&\leq &\norm{\p_q^{\beta_1+1}\p_p^{\beta_2}
  h}_2 \leq L_1^{\beta_1}L_2^{\beta_2}[(\abs{\beta}-1)!]^s.
\end{eqnarray*}
Thus   we obtain,
using Lemma
  \ref{+stab}-\reff{conc1} with $u=h_{p}$, $w=h_{p}$ and $v=h_{qp}$,
\begin{eqnarray*}\label{a}
        \norm{\p^{\tilde\alpha}(h_{p} h_q h_{qp})}_0\leq c_8 L_1^{\tilde\alpha_1-1}L_2^{\tilde\alpha_2}[(\abs{\tilde\alpha}-2)!]^s=c_8 L_1^{\alpha_1-1}L_2^{\alpha_2-1}[(\abs{\alpha}-3)!]^s
        \leq c_8 L_1^{\alpha_1}L_2^{\alpha_2-1}[(\abs{\alpha}-2)!]^s.
      \end{eqnarray*}
Similarly,  using Lemma
  \ref{+stab}-\reff{conc2} with $u=h_{p}$, $v=w=h_{qp}$, gives
 \begin{eqnarray*}\label{a}
        \norm{\p^{\tilde\alpha}(h_{p} h_{qp} h_{qp})}_0\leq c_9 L_1^{\tilde\alpha_1}L_2^{\tilde\alpha_2}[(\abs{\tilde\alpha}-1)!]^s
        =c_9 L_1^{\alpha_1}L_2^{\alpha_2-1}[(\abs{\alpha}-2)!]^s,
      \end{eqnarray*}
while  using Lemma
  \ref{+stab}-\reff{conc3} with $u=h_{p}$, $v=h_{q}$ and $w=h_{qpp}$ gives
 \begin{eqnarray*}\label{a}
       \norm{\p^{\tilde\alpha}(h_{p} h_{p} h_{qpp})}_0\leq c_{10} L_1^{\tilde\alpha_1}L_2^{\tilde\alpha_2}[(\abs{\tilde\alpha}-1)!]^s
        =c_{10} L_1^{\alpha_1}L_2^{\alpha_2-1}[(\abs{\alpha}-2)!]^s.
      \end{eqnarray*}
Combining the above inequalities, we have,  in view of \reff{d},
      \begin{eqnarray*}
       2 \norm{\p^{\alpha}(h_p h_q h_{qp})}_0\leq 2(c_8+c_9+c_{10}) L_1^{\alpha_1}L_2^{\alpha_2-1}[(\abs{\alpha}-2)!]^s.
      \end{eqnarray*}
   The treatment for the term $\norm{\p^{\alpha}(h_p^2h_{qq})}_{0}$ is
   completely the same as above, so we have
      \begin{eqnarray*}
        \norm{\p^{\alpha}(h_p^2h_{qq})}_{0}\leq c_{11} L_1^{\alpha_1}L_2^{\alpha_2-1}[(\abs{\alpha}-2)!]^s.
      \end{eqnarray*}
      Combining the above two estimates, we choose
      $c_2=2(c_8+c_9+c_{10})+c_{11}$ to complete the proof of the lemma.
    \end{proof}

      \begin{lemma}\label{lem403}
     Let $\gamma\in G^s([p_0,0])$. We have, for $\alpha=(\alpha_1,\alpha_2)=(m-n,n)$ with
     $1\leq n\leq m$,
 \begin{equation*}
        \norm{\p^{\alpha}\inner{\gamma h_p^3}}_{0}\leq c_3 L_1^{\alpha_1-2}L_2^{\alpha_2} [(\abs{\alpha}-2)!]^{s}.
      \end{equation*}
    \end{lemma}

\begin{proof}[Proof of the lemma]
As for $\norm{\p^{\xi} h_q}_{0}$ in \reff{10},   we have by induction
      \begin{equation*}
        \forall~  \beta\leq \alpha,~\abs{\beta}\geq 3,\quad \norm{\p^{\beta} h_p}_0\leq L_1^{\beta_1-2}L_2^{\beta_2}\com{(|\beta|-3)!}^s.
      \end{equation*}
Thus using Lemma \ref{+stab}-\reff{conc4} with $H_i=L_i, i=1,2$,
$u=v=w=h_p$,  we deduce that $c_*^{-1}L_1 h_p^3\in\mathcal A_1$, that is,
      \begin{equation}\label{betahp3}
        \forall~\beta\leq \alpha,~\abs{\beta}\geq 2,\quad
        \norm{\p^\beta \inner{ h_p^3}}_0 \leq c_* L_1^{\beta_1-2}L_2^{\beta_2}[(\abs{\beta}-2)!]^s.
      \end{equation}
On the other hand, since  $\gamma(p)\in G^s([p_0, 0])$,  then
using  \reff{GevCon}
 gives
\begin{eqnarray*}
   \forall~\abs{\beta}\geq 3,\quad \ \norm{\p^\beta \gamma}_0\leq
\left\{
\begin{array}{lll}
0,&\quad  \beta_1\geq 1,\\
 \norm{\p_p^{\beta_2} \gamma}_0 \leq M^{\beta_2+1}(\beta_2!)^s \leq
 \tilde M^{\beta_2}[(\beta_2-3)!]^s,
 &\quad \beta_1=0,
\end{array}
\right.
\end{eqnarray*}
where $\tilde M$ in the last inequality is a constant
depending only on  $M$ and $s$.  Thus if we choose $L_1, L_2$ in such a
way that
\begin{equation}\label{l1l24}
L_2\geq L_1 \tilde M,
\end{equation}
then we have
 \begin{equation}\label{gambeta3}
 \forall~\beta\leq \alpha,~\abs{\beta}\geq 3,\quad\norm{\p^\beta \gamma}_0\leq  L_1^{\beta_1-2}L_2^{\beta_2}[(\abs{\beta}-3)!]^s.
\end{equation}
Now we write
\begin{eqnarray*}
 \norm{\p^{ \alpha}( \gamma h_p^3)}_0\leq \sum_{\abs\beta\leq
   \abs\alpha}\frac{\abs\alpha !}{\abs\beta!(\abs\alpha-\abs\beta)!}\norm{\p^{ \beta}\gamma }_0 \norm{\p^{ \alpha-\beta}(  h_p^3)}_0.
\end{eqnarray*}
This together with \reff{betahp3} and \reff{gambeta3} allows us to  argue as the
treatment of $J_1-J_3$ in Lemma \ref{lem401},  to conclude
\begin{eqnarray*}
 \norm{\p^{ \alpha}( \gamma h_p^3)}_0\leq c_{12}
 L_1^{\alpha_1-2}L_2^{\alpha_2}[(|\alpha|-2)!]^s.
\end{eqnarray*}
Thus the desired estimate follows  if choosing $c_3=c_{12}$.  The proof
is thus complete.
\end{proof}

We now continue the proof of Proposition \ref{prophpq}.  Combining \reff{5}
and the conclusions in the previous three lemmas, Lemma \ref{lem401}-Lemma
\ref{lem403},  we get
\begin{eqnarray*}
  \begin{split}
   \norm{\p_q^{m-n}\p_p^{n+2} h}_{0}&\leq \inner{(c_1+c_3)L_1^{-1}+c_2L_1L_2^{-1}}L_1^{\alpha_1-1}L_2^{\alpha_2} [(\abs{\alpha}-2)!]^{s}\\
   &\leq \frac{1}{2}L_1^{\alpha_1-1}L_2^{\alpha_2} [(\abs{\alpha}-2)!]^{s},
   \end{split}
\end{eqnarray*}
where in the last inequality we chose
\begin{equation}\label{Lgeq}
L_1\geq 4(c_1+c_3) , \quad L_2\geq 4c_2L_1.
\end{equation}
Then
we get the desired estimate \reff{+hqmn},  and thus the validity of
$(F_{m,n})$ and $(F_{m})$.
Summarizing  the relations \reff{l1l21}, \reff{l1l22},
\reff{l1l23}, \reff{l1l24} and \reff{Lgeq},  we can choose
\begin{eqnarray}\label{L1 L2}
  L_1\geq \max\Big\{L, \norm{h}_4+1,  4(c_1+c_3)\Big\} ~~ {\rm and}~~ L_2\geq \inner{8+4c_2+\tilde M}L_1,
\end{eqnarray}
with $\tilde M$ the  constant appearing in \reff{l1l24},
to complete the proof of Proposition \ref{prophpq}.
\end{proof}

The rest of this section is devoted to
\begin{proof}[Proof of Lemma \ref{+stab}]  To  simplify
the notations, we   use $a_j,j\geq 1,$ to denote different
suitable harmless constants larger than $1$,   which depend only on
the dimension,  but are independent of  the order $\alpha_0$ of
derivative.

\reff{conc1}
Assume $u\in\mathcal A_2$ and $v\in\mathcal A_1$.  By Leibniz formula we have, for any  $\alpha\leq \alpha_0$ with
$\abs\alpha\geq 2$,
\begin{eqnarray*}
  \p^\alpha\inner{uv}=\sum_{0\leq \beta\leq\alpha} {\alpha\choose \beta}
   \inner{ \p^{\beta}u}   \inner{ \p^{\alpha-\beta}v}.
\end{eqnarray*}
Then
\begin{eqnarray*}
  \norm{\p^\alpha\inner{uv}} \leq  \sum_{0\leq\beta\leq\alpha} \frac{\abs\alpha!}{\abs\beta !\abs{\alpha-\beta}!}
   \norm{ \p^{\beta}u}_0 \norm{
       \p^{\alpha-\beta}v}_0
= I_1+I_2+I_3,
\end{eqnarray*}
with
\begin{eqnarray*}
       I_1&=&  \sum_{\stackrel{0\leq\beta\leq\alpha}{\abs\beta\leq 2}}  \frac{\abs\alpha!}{\abs\beta !\abs{\alpha-\beta}!}
   \norm{\p^{\beta}u}_0 \norm{
       \p^{\alpha-\beta}v}_0,\\
   I_2&=& \sum_{\stackrel{0\leq\beta\leq\alpha}{3\leq \abs\beta\leq \abs\alpha-2}}  \frac{\abs\alpha!}{\abs\beta !\abs{\alpha-\beta}!}
   \norm{ \p^{\beta}u}_0 \norm{
       \p^{\alpha-\beta}v}_0,\\
I_3&=&  \sum_{\stackrel{0\leq\beta\leq\alpha}{\abs\beta\geq\abs\alpha-1}}  \frac{\abs\alpha!}{\abs\beta !\abs{\alpha-\beta}!}
   \norm{ \p^{\beta}u}_0 \norm{
       \p^{\alpha-\beta}v}_0.
\end{eqnarray*}
Since $H_2\geq H_1$, direct computation shows that there exists $a_1>1$ such that
   \begin{eqnarray*}
    I_1+I_3
    \leq a_1 \inner{\norm{u}_2+\norm{v}_1+1}^2 H_1^{\alpha_1-1}H_2^{\alpha_2}[(\abs\alpha-2)!]^s.
   \end{eqnarray*}
For $I_2$, which appears only when $\abs{\alpha}\geq 5$, we have
\begin{eqnarray*}
    I_2&\leq&  \sum_{\stackrel{0\leq\beta\leq\alpha}{3\leq \abs\beta\leq \abs\alpha-2}}\frac{\abs\alpha!}{\abs\beta !\abs{\alpha-\beta}!}
     H_1^{\beta_1-2}H_2^{\beta_2}\inner{(\abs\beta-3)!}^sH_1^{\alpha_1-\beta_1-1}H_2^{\alpha_2-\beta_2}\com{(\abs{\alpha}- \abs{\beta}-2)!}^s\\
  &\leq& a_2 \sum_{\stackrel{0\leq\beta\leq\alpha}{3\leq \abs\beta\leq \abs\alpha-2}}\frac{\abs\alpha!}{\abs\beta^{3} \abs{\alpha-\beta}^{2}}
     H_1^{\alpha_1-3}H_2^{\alpha_2}\com{(\abs\beta-3)!}^{s-1}\com{(\abs{\alpha}- \abs{\beta}-2)!}^{s-1}\\
  &\leq& a_2 H_1^{\alpha_1-3}H_2^{\alpha_2} \sum_{\stackrel{0\leq\beta\leq\alpha}{3\leq \abs\beta\leq \abs\alpha-2}} \frac{\abs\alpha!}{\abs\beta^{3} \abs{\alpha-\beta}^{2}}
     \com{(\abs{\alpha}-5)!}^{s-1}\\
  &\leq& a_2H_1^{\alpha_1-1}H_2^{\alpha_2}\com{(\abs\alpha-2)!}^{s} \sum_{\stackrel{0\leq\beta\leq\alpha}{3\leq \abs\beta\leq \abs\alpha-2}} \frac{\abs\alpha^{2}}{ \abs\beta^{3}\inner{\abs\alpha-\abs\beta}^{2}}\\
  &\leq& a_3 H_1^{\alpha_1-1}H_2^{\alpha_2}\com{(\abs\alpha-2)!}^{s}.
\end{eqnarray*}
In view of the
estimates for $I_1,I_2$ and $I_3$, we have,  for any  $\alpha\leq
\alpha_0$  with $\abs\alpha\geq 2$,
\begin{equation}\label{uvnorm}
  \norm{\p^{\alpha}(uv)}_0\leq a_4\inner{\norm{u}_2+\norm{v}_1+1}^2 H_1^{\alpha_1-1}H_2^{\alpha_2} \com{(\abs\alpha-2)!}^{s}
\end{equation}
by choosing $a_4= (a_1+ a_3)$.

If additionally $w\in\mathcal A_2$,  then applying  the above
arguments to the two functions $w$   and 
$$a_4^{-1}\inner{\norm{u}_2+\norm{v}_1+1}^{-2}uv$$ 
which lies in $\mathcal A_1$ due to
\reff{uvnorm}, gives
\begin{eqnarray*}
  \begin{split}
    \norm{\p^{\alpha}(uv w)}_0&\leq a_4^2\inner{\norm{u}_2+\norm{v}_1+1}^2\inner{\norm{w}_2+\norm{uv}_1+1}^2 H_1^{\alpha_1-1}H_2^{\alpha_2} \com{(\abs\alpha-2)!}^{s}\\
    &\leq a_4^2\inner{\norm{u}_2+\norm{v}_1+\norm{w}_2+1}^6 H_1^{\alpha_1-1}H_2^{\alpha_2} \com{(\abs\alpha-2)!}^{s}.
  \end{split}
\end{eqnarray*}
Thus the conclusion \reff{conc1} follows if
we choose $c_*\geq a_4^2\inner{\norm{u}_2+\norm{v}_1+\norm{w}_2+1}^6$.

\reff{conc2} Now assume $u\in \mathcal A_2$ and $v,w\in\mathcal A_1$.
Firstly note from \reff{uvnorm} that $$a_4^{-1}\inner{\norm{u}_2+\norm{v}_1+1}^{-2}uv\in\mathcal{A}_1.$$
We can use the same arguments as above to the two functions
$a_4^{-1}\inner{\norm{u}_2+\norm{v}_1+1}^{-2}uv$ and $w$; this gives, for any $\alpha\leq \alpha_0$ with
$\abs\alpha\geq 1,$
\begin{eqnarray*}
    a_4^{-1}\inner{\norm{u}_2+\norm{v}_1+1}^{-2}\norm{\p^{\alpha}\inner{uv w}}_0
  &\leq& a_5\inner{\norm{uv}_1+\norm{w}_1+1}^2 H_1^{\alpha_1}H_2^{\alpha_2} \com{(\abs\alpha-1)!}^{s}\\
  &&+a_5 H_1^{\alpha_1-1}H_2^{\alpha_2}\com{(\abs\alpha-2)!}^{s} \sum_{\stackrel{0\leq\beta\leq\alpha}{2\leq \abs\beta\leq \abs\alpha-2}} \frac{\abs\alpha^{2}}{ \abs\beta^{2}\inner{\abs\alpha-\abs\beta}^{2}}\\
  &\leq& a_6\inner{\norm{uv}_1+\norm{w}_1+1}^2 H_1^{\alpha_1}H_2^{\alpha_2} \com{(\abs\alpha-1)!}^{s},
\end{eqnarray*}
where the last inequality using the estimate 
\[
\sum_{\stackrel{0\leq\beta\leq\alpha}{2\leq \abs\beta\leq
    \abs\alpha-2}} \frac{\abs\alpha^{2}}{
  \abs\beta^{2}\inner{\abs\alpha-\abs\beta}^{2}}\leq 8\pi^2.
\]
 Accordingly,
\begin{eqnarray*}
  \norm{\p^{\alpha}\inner{uv w}}_0\leq a_4a_6\inner{\norm{u}_2+\norm{v}_1+\norm{w}_1+1}^{6}H_1^{\alpha_1}H_2^{\alpha_2} \com{(\abs\alpha-1)!}^{s}.
\end{eqnarray*}
Thus the
conclusion follows if we choose $c_* \geq a_4a_6\inner{\norm{u}_2+\norm{v}_1+\norm{w}_1+1}^{6}$.

\reff{conc3} Now consider the case when  $u, v\in \mathcal A_2$ and $w\in\mathcal A_0$.
Similarly we can first use the same arguments as in
\reff{conc1},   to
obtain $$a_{7}^{-1}\inner{\norm{u}_2+\norm{w}_0+1}^{-2} uw\in\mathcal{A}_0,$$ and then  repeat the
arguments to the two functions $a_{7}^{-1}\inner{\norm{u}_2+\norm{w}_0+1}^{-2} uw$ and $v$ to conclude
\[
a_{7}^{-1}a_{8}^{-1}\inner{\norm{u}_2+
\norm{v}_2+\norm{w}_0+1}^{-6} uvw\in\mathcal{A}_0.
\]
 The conclusion \reff{conc3} follows
by choosing $c_*\geq a_{7}a_{8}\inner{\norm{u}_2+
\norm{v}_2+\norm{w}_0+1}^{6}$.

\reff{conc4}
Assume $u,v,w\in\mathcal{A}_2$. Similarly, we can first argue as in \reff{conc1} to show that
\begin{eqnarray*}
 \forall~\abs{\alpha}\geq 2,\quad \norm{\p^\alpha (uv)}_0\leq a_9\inner{\norm{u}_2+\norm{v}_2+1}^2 H_1^{\alpha_1-2}H_2^{\alpha_2}[(\abs{\alpha}-2)!]^s,
\end{eqnarray*}
and then repeat the arguments to $a_9^{-1}\inner{\norm{u}_2+\norm{v}_2+1}^{-2}uv$ and $w$ to derive
\begin{eqnarray*}
  \forall~\abs{\alpha}\geq 2,\quad \norm{\p^\alpha (uvw)}_0\leq a_9a_{10}\inner{\norm{u}_2+\norm{v}_2+\norm{w}_2+1}^6 H_1^{\alpha_1-2}H_2^{\alpha_2}[(\abs{\alpha}-2)!]^s.
\end{eqnarray*}
Thus the conclusion \reff{conc4} follows by choosing $c_*\geq a_9a_{10}\inner{\norm{u}_2+\norm{v}_2+\norm{w}_2+1}^6$.

Finally, the conclusion of Lemma \ref{+stab}  follows by choosing
\[
  c_*\geq (a_4^2+a_4a_6+a_7a_8+a_9a_{10})\inner{\norm{u}_2+\norm{v}_2+\norm{w}_2+1}^6
\]
with $a_j$ the constants depending only on the dimension.
The proof is thus complete.
\end{proof}

\section{Regularity of water waves with surface tension}\label{sec5}
Adding the effects of surface tension in the free boundary problem \reff{EquPsi1}-\reff{EquPsi4} introduces higher-order derivative into the boundary condition. That is the equation \reff{EquPsi2} is replaced by
\begin{equation}\label{EquPsi22}
  \abs{\nabla
    \psi}^2+2g(y+d)-2\sigma\frac{\eta_{xx}}{(1+\eta_x^2)^{\frac{3}{2}}}=Q,\qquad
  y=\eta(x),  \tag{3b$'$}
  \end{equation}
where $\sigma>0$ is the coefficient of surface tension. Correspondingly, the equation \reff{Equh2} becomes
\begin{equation}\label{Equh22}
  1+h_q^2+(2gh-Q)h_p^2-2\sigma\frac{h_p^2h_{qq}}{(1+h_q^2)^{\frac{3}{2}}}=0,\quad
  {\rm on~~} p=0. \tag{4b$'$}
\end{equation}

Proven in this section is the regularity property of all
the streamlines and stream function of water waves with surface
tension.

\begin{theorem}\label{th3}
   Consider the free boundary problem \reff{EquPsi1}-\reff{EquPsi4} with
   \reff{EquPsi2} replaced by  above \reff{EquPsi22}.  Suppose $\gamma\in
   C^{0,\mu}([p_0,0])$ with $\mu$ and $p_0$ given. Then
each streamline including the free surface $y=\eta(x)$ is a real-analytic curve.   If, in addition,
$\gamma \in G^s([p_0, 0])$ with $s\geq 1$,   then $\psi(x,y) \in
G^s(\bar \Omega )$; in particular if $s=1$, i.e., $\gamma$ is analytic in $[p_0,
0]$, then the stream function $\psi(x,y)$ is analytic in $\bar\Omega$.
\end{theorem}

\begin{proof}

As before we only prove the corresponding regularity for height
function $h$ of the system \reff{Equh1}-\reff{Equh3} with \reff{Equh2}
replaced by above \reff{Equh22}.   Since the arguments are nearly the same as those in the absence of surface tension (Section
\ref{sec3} and Section \ref{sec4}),  we shall only give a sketch and indicate how to modify the analysis as
adding the higher-order derivative due to surface tension.

Repeating the arguments in Section \ref{sec4},  we can derive  the
second statement in Theorem \ref{th3},  without any difference. So we
only need to prove the first statement on the analyticity of streamlines,
where the main difference from Section \ref{sec3} occurs.  As in
Remark \ref{remReg} we may assume $\p_q^k h\in C^{2,\mu}(\bar R)$ for
any $k\in\N$.   Taking $m^{th}$-order derivative with respect to the $q$-variable on both sides of the equation \reff{Equh22} shows that the second equation in \reff{oprtAB} becomes
\begin{eqnarray*}
  \tilde B(h)[\p_q^m h]=\tilde\varphi_1+\tilde\varphi_2,\quad {\rm on}~ p=0,
\end{eqnarray*}
with the operator
\begin{eqnarray*}
  \tilde B(h)=2\sigma\frac{h_p^2}{(1+h_q^2)^{\frac{3}{2}}}\p_q^2
\end{eqnarray*}
and the right-hand side
\begin{eqnarray*}
  \tilde\varphi_1=\p_q^m h_q^2+\p_q^m \inner{(2gh-Q)h_p^2},\quad \tilde\varphi_2=-2\sigma \sum_{1\leq n\leq m}(\p_q^{m-n} h_{qq})\inner{\p_q^n \frac{h_p^2}{(1+h_q^2)^{\frac{3}{2}}}}.
\end{eqnarray*}
The first and third equations in \reff{oprtAB} remain unchanged. Then,
as before, our aim is to show that the corresponding estimate as  in Proposition \ref{propassu} holds, that is, there exists a  constant $\tilde L\geq 1$ such
that for any $m\geq2,$
\begin{eqnarray*}
(\tilde E_m): \quad\quad \norm{\p^m_q h}_{2,\mu}\leq
\tilde L^{m-1}(m-2)!,
\end{eqnarray*}
and to this end the main point is to show that $(\tilde E_m)$ holds under the assumption that for
any $j$ with $2\leq j\leq m-1$, the following estimate
\begin{eqnarray}\label{tildIne}
  (\tilde E_j):\quad\quad \norm{\p^j_q h}_{2,\mu}\leq
  L^{j-1}(j-2)!
\end{eqnarray}
is already valid.

Since $h\in C^{2,\mu}(\bar R)$, the coefficient
$\frac{h_p^2}{(1+h_q^2)^{\frac{3}{2}}}$ of the operator $\tilde B(h)$
is in $C^{1,\mu}(\bar R)$. Moreover by the induction assumption
\reff{tildIne},
$\tilde\varphi_1$ and $\tilde\varphi_2$ are in $C^{0,\mu}(\bar R)$. Furthermore it has been shown in \cite{HenryJmfm} that the operator $\tilde B(h)$ satisfies the complementing condition in the sense of \cite{MR0125307}. As a result, we can apply the Schauder estimate in \cite{MR0125307} to conclude
\begin{eqnarray}\label{shauder2}
  \norm{\p_q^m h}_{2,\mu}\leq \tilde {\mathcal C}\inner{\norm{\p_q^m h}_0+\sum_{i=1}^2\norm{f_i}_{0,\mu}+\norm{\tilde\varphi_1}_{0,\mu}+\norm{\tilde\varphi_2}_{0,\mu}},
\end{eqnarray}
with $\tilde {\mathcal C}$ a constant independent of $m$,   and $f_i$, $i=1,2$, defined in
\reff{def f1}-\reff{def f2}.   As for the first three terms on the right
hand side,  we can use  the similar
arguments as in Section \ref{sec analyticity} without any additional
difficulty,   to  conclude
\[
\norm{\p_q^m
  h}_0+\sum_{i=1}^2\norm{f_i}_{0,\mu}+\norm{\tilde\varphi_1}_{0,\mu}\leq
\tilde C_1 \tilde L^{m-2}(m-2)!,
\]
with  $\tilde C_1$ a constant independent of  $m$.   It remains to   estimate
$\norm{\tilde\varphi_2}_{0,\mu}$, and show that for some constant
$\tilde C_2$,
\begin{equation}\label{lain}
\norm{\tilde\varphi_2}_{0,\mu}\leq
\tilde C_2 \tilde L^{m-2}(m-2)!.
\end{equation}
To do so we need the following
lemma, whose proof is postponed to the end of this section.

\begin{lemma}\label{+stab+}

Let $C_*\geq 1$  be the constant given in Lemma \ref{stab},   and let
$k_0\in\N$ with
$k_0\geq 3$.  Suppose  $\p_q^k u\in
C^{0,\mu}(\bar R)$ for any $k\leq k_0$.  If there
exist  two constants $C_0$ and $ \tilde H$ satisfying
\begin{equation}\label{biglittle}
  C_0\geq C_*\inner{2\big\|(1+u^2)^{-1}\big\|_{2,\mu}+2\big\|(1+u^2)^{-3/2}\big\|_{2,\mu}+\big\|\p_q(u^2)\big\|_{2,\mu}+1}^6
\end{equation}
and
\begin{equation}\label{+biglittle}
\tilde H\geq  2 C^2_0+\big\|\p_q^3\inner{(1+u^2)^{-1}}
 \big\|_{0,\mu}+\big\|\p_q^3\inner{(1+u^2)^{-3/2}}
 \big\|_{0,\mu},
\end{equation}
such that
\begin{equation}\label{+condition3}
 \forall~ 3\leq k\leq k_0,  \quad  \|\p_q^k (u^2)\|_{0,\mu}
  \leq C_0 \tilde H^{k-2}(k-3)!,
\end{equation}
then
\begin{eqnarray}\label{++conclusion}
\forall~ 3\leq k\leq k_0,  \quad\big\|\p_q^k \inner{(1+u^2)^{-3/2}} \big\|_{0,\mu}\leq C_0^2  \tilde H^{k-2}(k-3)!.
\end{eqnarray}

\end{lemma}

 We now use the above lemma to prove \reff{lain}. Observe
\begin{eqnarray*}
  \norm{\tilde\varphi_2}_{0,\mu}\leq 2\sigma \sum_{1\leq n\leq m}\norm{\p_q^{m-n} h_{qq}}_{0,\mu}\norm{\p_q^n (h_p^2(1+h_q^2)^{-3/2})}_{0,\mu}.
\end{eqnarray*}
Using the induction assumption \reff{tildIne}, we have
\begin{eqnarray}\label{In+hqq}
 \forall~1\leq n\leq m-2,\quad \norm{\p_q^{m-n} h_{qq}}_{0,\mu}\leq
 \norm{\p_q^{m-n} h}_{2,\mu}\leq \tilde L^{m-n-1}(m-n-2)!,
\end{eqnarray}
and
\[
   \forall~3\leq n\leq m,\quad \norm{\p_q^{n} h_q}_{0,\mu}\leq
 \norm{\p_q^{n-1} h}_{2,\mu}\leq \tilde L^{n-2}(n-3)!.
\]
This last inequality along with  Lemma \ref{stab}, with $u_1=u_2=h_q, u_3=1$,  implies
\begin{eqnarray*}
  \forall~3\leq n\leq m,\quad \norm{\p_q^n h^2_q}_{0,\mu}\leq C_*\inner{2\norm{h_q}_{3,\mu}+1}^6\tilde L^{n-2}(n-3)!\leq 
 C_0\tilde L^{n-2}(n-3)!,
\end{eqnarray*}
where in the last inequality  we choose 
\[
    C_0\geq C_*\inner{2\big\|(1+h_q^2)^{-1}\big\|_{2,\mu}+2\big\|(1+h_q^2)^{-3/2}\big\|_{2,\mu}+\big\|\p_q(h_q^2)\big\|_{2,\mu}+2\norm{h_q}_{3,\mu}+1}^6.
\]
Now choosing   $\tilde L$ in such a way that
$$\tilde L\geq 2C^2_0 +\big\|\p_q^3\inner{(1+h_q^2)^{-1}}
 \big\|_{0,\mu}+\big\|\p_q^3\inner{(1+h_q^2)^{-3/2}}
 \big\|_{0,\mu},$$
then applying the above Lemma \ref{+stab+}, with $k_0=m$, $u=h_q$
and $\tilde H=\tilde L$,   we have,
\begin{eqnarray*}
  \forall ~ 3\leq n\leq m,\quad\norm{\p_q^n \inner{(1+h^2_q)^{-3/2}}}_{0,\mu}\leq C_0^2\tilde L^{n-2}(n-3)!.
\end{eqnarray*}
This along with \reff{In+hqq} allows us to  argue as in the
proof of  \reff{step2} for $\norm{f_1}_{0,\mu}$,   to  obtain \reff{lain}.
Thus choosing $$\tilde L\geq \tilde {\mathcal C}(\tilde C_1+\tilde C_2)+2C^2_0+\big\|\p_q^3\inner{(1+h_q^2)^{-1}}
 \big\|_{0,\mu}+\big\|\p_q^3\inner{(1+h_q^2)^{-3/2}}
 \big\|_{0,\mu},$$
we
get the validity of $(\tilde E_m)$.  The proof of Theorem \ref{th3} is thus
complete.
\end{proof}

The rest is occupied  by

\begin{proof}[Proof of Lemma \ref{+stab+}]
As a preliminary step we first use  induction to prove
\begin{eqnarray}\label{+conclusion}
 \forall~3\leq k\leq k_0,\quad \big\|\p_q^k \inner{(1+u^2)^{-1}}
 \big\|_{0,\mu}\leq  \tilde H^{k-1} (k-2)!.
\end{eqnarray}
In fact
\reff{+conclusion} obviously holds when $k=3$  due to \reff{+biglittle}.
Now assuming  $k\geq 4$ and that
\begin{equation}\label{indel}
\forall~ 3\leq j \leq k-1,   \quad \norm{\p_q^j
  \inner{(1+u^2)^{{-1}}}}_{0,\mu} \leq
\tilde H^{j-1}(j-2)!,
\end{equation}
we  show that the above inequality still holds for
$k$ with $k\leq k_0$. To do so,  write
\begin{equation}\label{betatilde}
    \p_q^k \inner{(1+u^2)^{{-1}}}=  -\p_q^{k-1}
    \inner{(1+u^2)^{-1}(1+u^2)^{-1} \p_q
      (u^2)}.
\end{equation}
Next we intend to apply Lemma \ref{stab} to prove that the right-hand  side of the above equation satisfies
\begin{eqnarray}\label{finnum}
   \norm{\p_q^{k-1}
    \inner{(1+u^2)^{-1}(1+u^2)^{-1} \p_q
      (u^2)}}_{0,\mu}\leq C_0^2 \tilde H^{k-2}(k-3)!.
\end{eqnarray}
In fact  for  any $j$ with $3\leq j\leq k-1$, one has
by \reff{+condition3},
\[
\|\p_q^j \p_q(u^2)\|_{0,\mu}
  \leq C_0\tilde H^{j-1}(j-2)!,
\]
and by the induction assumption \reff{indel}
\[
  \norm{\p_q^j
   \inner{ (1+u^2)^{-1} }}_{0,\mu} \leq
\tilde H^{j-1}(j-2)!.
\]
The above two estimates allow us to use Lemma \ref{stab},
with $\ell=1$, $u_1=u_2= (1+u^2)^{-1}$ and $u_3=C_0^{-1}\p_q(u^2)$, to obtain
\begin{eqnarray*}
&&\norm{\p_q^{k-1}
    \inner{(1+u^2)^{-1}(1+u^2)^{-1} \p_q
      (u^2)}}_{0,\mu}\\
      &\leq&C_0 C_*
  \inner{2\norm{(1+u^2)^{-1}}_{2,\mu}+C_0^{-1}\norm{\p_q(u^2)}_{2,\mu}+1}^6 \tilde
  H^{k-2}(k-3)!\\
&\leq& C_0^2 \tilde H^{k-2}(k-3)!,
\end{eqnarray*}
the last inequality using \reff{biglittle}. Thus the
desired inequality \reff{finnum} follows. As a result, combining \reff{betatilde} and
  \reff{finnum}, we conclude
\begin{eqnarray*}
    \norm{\p_q^{k}
    \inner{ (1+u^2)^{{-1}}}}_{0,\mu} \leq  C_0^2 \tilde H^{k-2}(k-3)! \leq   \tilde H^{k-1}(k-2)!,
\end{eqnarray*}
where the last inequality holds because $\tilde H\geq 2C_0^2$
due to \reff{+biglittle}.
We have proven \reff{+conclusion}.

Now  we prove \reff{++conclusion}, which
obviously holds when $k=3$ in view of  \reff{+biglittle}.
Now assuming  $k\geq 4$ and that
\begin{equation}\label{+indel}
\forall~ 3\leq j \leq k-1,   \quad \norm{\p_q^j
  \inner{(1+u^2)^{-3/2}}}_{0,\mu} \leq C^2_0
\tilde H^{j-2}(j -3)!,
\end{equation}
we  show the above equality still holds for
$k$ with $k\leq k_0$.  As before,  write
\begin{equation}\label{+betatilde}
    \p_q^k \inner{(1+u^2)^{{-3/2}}}=  -\p_q^{k-1}
    \inner{\frac{3  (1+u^2)^{-3/2}}{2}(1+u^2)^{-1} \p_q
      (u^2)}.
\end{equation}
Then the estimate in \reff{+indel} for $j=k$  will hold if we can show
that
\begin{eqnarray}\label{lastfinalIn}
  \frac{3 }{2} \norm{\p_q^{k-1}
    \inner{ (1+u^2)^{-3/2}(1+u^2)^{-1} \p_q
      (u^2)}}_{0,\mu}\leq  C^2_0
\tilde H^{k-2}(k-3)!.
\end{eqnarray}
Again we next intend to apply Lemma \ref{stab} to prove the above
estimate . For any $j $ with $3\leq j\leq k-1$,  one has
by \reff{+condition3}
\[
\|\p_q^j \p_q(u^2)\|_{0,\mu}
  \leq C_0 \tilde H^{j-1}(j -2)!,
\]
and by the induction assumption \reff{+indel}
\[
  \frac{3}{2}\norm{\p_q^j
   \inner{ (1+u^2)^{-3/2} }}_{0,\mu} \leq \frac{3}{2}  C^2_0
\tilde H^{j-2}(j-3)! \leq
\tilde H^{j-1}(j-2)!,
\]
the last inequality using the fact that $\tilde H\geq 2C_0^2\geq
3C^2_0/2$ due to  \reff{+biglittle}.
The above two estimates along with \reff{+conclusion} allow us to use Lemma \ref{stab},
with $\ell=1$, $$u_1= \frac{3}{2}(1+u^2)^{-3/2}, ~~u_2= (1+u^2)^{-1}, \quad u_3=C_0^{-1}\p_q(u^2);$$
this gives
\begin{eqnarray*}
&&\frac{3 }{2} \norm{\p_q^{k-1}
    \inner{ (1+u^2)^{-3/2}(1+u^2)^{-1} \p_q
      (u^2)}}_{0,\mu}\\
&\leq&C_0 C_*\inner{\frac{3 }{2}\norm{(1+u^2)^{-3/2}}_{2,\mu}+\norm{(1+u^2)^{-1}}_{2,\mu}+\norm{\p_q
      (u^2)}_{2,\mu}+1}^6
\tilde H^{k-2}(k-3)!\\
&\leq &C_0^2 \tilde H^{k-2}(k-3)!,
\end{eqnarray*}
where the last inequality holds because of \reff{biglittle}.
 Thus the desired estimate \reff{lastfinalIn} follows.
We have proven \reff{++conclusion}, completing the
proof of Lemma \ref{+stab+}.
\end{proof}

\bigskip
\noindent{\bf Acknowledgments}  The work is supported by
NSFC(11001207,11201355) and the RFDP (20100141120064).

\end{document}